\documentclass[a4paper]{amsart}

\usepackage[utf8]{inputenc}
\usepackage{enumitem}
\usepackage{setspace}
\usepackage{mathtools}
\usepackage{tikz-cd}
\usepackage{lipsum}
\usepackage{tikz}
\usepackage{mathrsfs}
\usetikzlibrary{arrows}
\usepackage{amssymb}
\usetikzlibrary{decorations.pathmorphing}
\usepackage{bbold}
\usepackage{amsmath,amsthm}
\usepackage{xspace}
\usepackage{faktor}
\usepackage{calrsfs}
\usepackage{verbatim}
\usepackage{quiver}
\usepackage[pdfa]{hyperref}
\usepackage{amsfonts}
\usepackage{graphicx}
\usepackage{eucal}
\usepackage{amscd}
\usepackage[all,2cell]{xy}
\UseAllTwocells

\newtheorem{teo}{Theorem}[section]
\newtheorem{cor}[teo]{Corollary}
\newtheorem{lem}[teo]{Lemma}
\newtheorem{defi}[teo]{Definition}

\newtheorem{prop}[teo]{Proposition}
\newtheorem{example}[teo]{Example}

\newdir{ m>}{{}*!/-3.5pt/@{}*!/-8pt/:(1,-.2)@^{>}*!/-8pt/:(1,+.2)@_{>}}

\def\pullback{
	\ar@{-}[]+R+<6pt,-1pt>;[]+RD+<6pt,-6pt>%
	\ar@{-}[]+D+<1pt,-6pt>;[]+RD+<6pt,-6pt>}

\newcommand{\zeros}{\mathcal{Z}}
\newcommand{\zideal}{\mathcal{N}_{\zeros}}
\newcommand{\op}[1]{#1^{\mathrm{op}}}
\newcommand{\tort}{\mathcal{T}}
\newcommand{\torf}{\mathcal{F}}
\newcommand{\torsione}{(\tort, \torf)}

\newcommand{\facte}{\mathbb{E}}
\newcommand{\factm}{\mathbb{M}}
\newcommand{\frecce}{\Arr(\mathcal{C})}

\newcommand{\boole}{\mathbf{Boole}}

\newcommand{\graffe}[1]{\{ #1 \}}
\newcommand{\spigolose}[1]{\langle #1 \rangle}

\newcommand{\inv}[1]{(#1)^{-1}}

\newcommand{\arrz}{\Arr(\zeros)}
\newcommand{\arrc}{\Arr(\mathcal{C})}
\newcommand{\ce}{\mathbb{E}}
\newcommand{\cm}{\mathbb{M}}
\newcommand{\Hf}{\F_1}
\newcommand{\Ht}{\T_1}
\newcommand{\Hv}{\V_1}
\newcommand{\Two}{\mathbf{2}}
\newcommand{\Ze}{\mathbb{Z}}

\DeclareMathOperator{\Hom}{Hom}
\DeclareMathOperator{\Arr}{Arr}
\DeclareMathOperator{\Zker}{\zeros-ker}
\DeclareMathOperator{\Zcoker}{\zeros-coker}
\DeclareMathOperator{\ArrZker}{\arrz-ker}
\DeclareMathOperator{\ArrZcoker}{\arrz-coker}

\def\mathsfdef#1{\expandafter\def\csname#1\endcsname{{\sf#1}}}
\mathsfdef{Z}
\mathsfdef{F}
\mathsfdef{G}
\mathsfdef{T}
\mathsfdef{S}
\mathsfdef{C}
\mathsfdef{i}
\mathsfdef{P}
\mathsfdef{D}
\mathsfdef{j}
\mathsfdef{V}

\begin{document}
	
	\title[Torsion Theories in a non-pointed Context II]{Torsion Theories in a non-pointed Context II: iterated torsion theories}
	
	\author[A. Cappelletti]{Andrea Cappelletti}
	\address[Andrea Cappelletti]{Dipartimento di Matematica, Universit\`{a} degli Studi di Salerno, Via Giovanni Paolo II 132, 84084 Fisciano (SA), Italy}
	\thanks{}
	\email{andreacappellettiphd@gmail.com}
	
	\author[A. Montoli]{Andrea Montoli}
	\address[Andrea Montoli]{Dipartimento di Matematica ``Federigo Enriques'', Universit\`{a} degli Studi di Milano, Via Saldini 50, 20133 Milano, Italy}
	\thanks{} \email{andrea.montoli@unimi.it}
	
	
	
	\begin{abstract}
		Generalizing a known result in the context of homological categories, we show that, in a regular protomodular category $\mathcal{C}$ equipped with a posetal, monocoreflective subcategory whose coreflector inverts monomorphisms, every torsion theory, in a non-pointed sense, satisfying suitable properties gives rise to a torsion theory in the category $\arrc$ of arrows in $\mathcal{C}$. Moreover, we show that the induced torsion theory in $\arrc$ satisfies the same good properties of the original one in $\mathcal{C}$, allowing to iterate the process. Examples of this situation can be found in protomodular varieties of universal algebras with more than one constant or in duals of elementary toposes.
	\end{abstract}
	
	\date{\today}
	
	\maketitle
	
	\section{Introduction}
	The categorical notion of torsion theory has been introduced in \cite{Dickson} for abelian categories and extended to general pointed categories in \cite{JT torsion th}. In the case of homological categories \cite{Bbbook}, namely pointed, regular and protomodular \cite{Bourn protomod} categories, it was shown in \cite{BG torsion th} that the notion of torsion theory is strongly connected to the ones of radical and of closure operator. The exploration of these connections with other important notions has been continued in \cite{Gran torsione}, where it is shown that torsion theories in a homological category, satisfying an additional condition (called Condition (N) there) concerning the fact that the torsion subobjects are characteristic in the sense of \cite{CM char}, correspond to stable factorization systems whose ``epimorphic'' part is made of cokernels. Moreover, the link between torsion theories and admissible Galois structures (in the sense of \cite{janelidze1990pure}) is widely explored in \cite{Gran torsione}; it is shown there, in particular, that torsion theories in a homological category whose reflector to the torsion-free subcategory is protoadditive give rise to Galois structures where the central and normal extensions have a simple description. The notion of protoadditive functor is a generalization, introduced in \cite{everaert2010homology}, of the classical notion of additive functor between additive categories.
	
	It was shown in \cite{EGVdL Hopf} that the reflector of a Birkhoff subcategory of a semi-abelian \cite{JMT semi-ab} category forms a Galois structure which allows to recover homology groups with coefficients in the Birkhoff subcategory by means of generalized Hopf formulae. These formulae are obtained using higher dimensional Galois structures built from the starting one formed by the Birkhoff reflector, replacing the ground category with the category of extensions (which, in the semi-abelian case, are just regular epimorphisms) in it and the reflective subcategory with the one of normal extensions with respect to the starting Galois structure. Iterating the process, one gets a sequence of derived Galois structures, that are used to compute the homology groups.
	
	In \cite{Gran torsione} these results of \cite{EGVdL Hopf} are extended to homological categories, replacing a Birkhoff reflector with the protoadditive reflector of a torsion-free subcategory. In order to do this, the authors first show that any torsion theory $\torsione$ in a homological category $\mathcal{C}$, satisfying Condition (N) and such that the torsion-free reflector is protoadditive, gives rise to a torsion theory $(\tort_1, \torf_1)$ in the category $\arrc$ of arrows in $\mathcal{C}$ still satisfying (N) and the protoadditivity of the reflector, allowing to iterate the process. Then, they restrict the torsion theory $(\tort_1, \torf_1)$ to the category of extensions in $\mathcal{C}$ (that, in this context, are effective descent morphisms), with the subcategory of normal extensions with respect to the associated Galois structure as torsion-free part. Again, this process can be iterated, giving rise to a sequence of derived torsion theories in the category of $n$-extensions, and to the associated Galois structures, from which homology groups are extracted.
	
	In recent years, several authors considered non-pointed versions of the notion of torsion theory. In \cite{GrandisJan torsion th} the authors introduced a notion of torsion theory in categories equipped with an ideal $\mathcal{N}$ of morphisms in the sense of C. Ehresmann \cite{Ehresmann}, using kernels and cokernels with respect to the ideal $\mathcal{N}$; they assumed that $\mathcal{N}$-kernels and $\mathcal{N}$-cokernels of identity morphisms exist, and they explored the relationship of this non-pointed notion of torsion theory with closure operators and factorization systems. A similar notion of torsion theory (called pretorsion theory there) was considered in \cite{Facchini2}, where kernels and cokernels (whose existence is not assumed) are considered with respect to the ideal of morphisms that factor through the intersection between the torsion subcategory and the torsion-free one. In \cite{CM torsion theories 1} we studied these non-pointed notions of torsion theory in the context of regular protomodular categories equipped with a posetal, monocoreflective subcategory $\zeros$ whose coreflector inverts monomorphisms, calling such a subcategory a class of zero objects. The name is justified by the fact that the objects of this subcategory play the role that, in the pointed case, the zero object plays. The main family of examples of such a context is given by regular protomodular categories with initial object and such that the unique morphism from the initial object to the terminal one is a regular epimorphism, and the class of zero objects is given by the regular quotients of the initial object. Concrete examples of such a situation are protomodular varieties of universal algebras with more than one constant, like unitary rings, Boolean algebras, Heyting algebras, MV-algebras, their topological models, as well as the dual of every elementary topos.
	
	In \cite{CM torsion theories 1} we generalized to our non-pointed context the results of \cite{Gran torsione} concerning the relationship between torsion theories, factorization systems, and Galois structures. In order to do that, we had to consider not only Condition (N) and protoadditivity of the torsion-free reflector, as in \cite{Gran torsione}, but also a few additional conditions that are automatically satisfied in the pointed case. Moreover, we exhibited several examples of torsion theories in our context satisfying all the needed conditions; these examples live in the categories of Heyting algebras, MV-algebras, in the duals of the toposes of simplicial sets and $M$-sets (where $M$ is a monoid), and in coslices of the category of abelian groups.
	
	The aim of the present work is to show that, also in our non-pointed context, every good torsion theory $\torsione$ in a category $\mathcal{C}$ with a class $\zeros$ of zero objects induces a torsion theory $(\tort_1, \torf_1)$ in the category $\arrc$ of arrows in $\mathcal{C}$, which is itself equipped with the class of zero objects $\arrz$ formed by arrows between objects in $\zeros$. Moreover, we show that all good properties of $\torsione$ are transferred to $(\tort_1, \torf_1)$. Showing that the properties that are trivially satisfied in the pointed case are transferred to $(\tort_1, \torf_1)$ is relatively easy, while the proofs that Condition (N) and protoadditivity of the torsion-free reflector pass to $(\tort_1, \torf_1)$ are quite different from the ones given in \cite{Gran torsione} for the pointed case, and far more complicated. Still, all the properties are transferred to the torsion theory in $\arrc$, allowing to iterate the process and get a sequence of torsion theories $(\tort_n, \torf_n)$ in $\Arr^n(\mathcal{C})$. Showing that these iterated torsion theories can be restricted to categories of $n$-extensions, in order to get homology groups from the associated Galois theories, is material for future work.
	
	\section{Non-pointed torsion theories}
	In this section we recall from \cite{CM torsion theories 1} the context in which we are interested, as well as the results concerning the connection between the notion of torsion theory and those of factorization system and Galois structure.
	
	\begin{defi}
		Let $\mathcal{C}$ be a category. A full subcategory $\zeros$ is a \emph{class of zero objects} if it is a posetal monocoreflective subcategory such that the coreflector inverts monomorphisms.
	\end{defi}
	
	Given a class $\zeros$ of zero objects in a category $\mathcal{C}$, we denote by $\Z$ the coreflector of $\zeros$ and, given an object $C$ of $\mathcal{C}$, we call $\Z(C)$ the \emph{zero part} of $C$.
	
	We observed in \cite[Remark 2.2]{Cappelletti} that such a class $\zeros$, when it exists, is unique. So, admitting a class of zero objects is a property of the category $\mathcal{C}$. The context we will mainly consider is the one of a regular category $\mathcal{C}$ with initial object, where we take as $\zeros$ the class of regular quotients of the initial object. Examples of this situation are elementary toposes, their duals, all varieties of universal algebras with at least one constant, and their topological models. \\
	
	Given a category $\mathcal{C}$ with a class $\zeros$ of zero objects, we denote by $\mathcal{N}_{\zeros}$ the class of arrows of $\mathcal{C}$ factorizing through an object of $\zeros$. It forms an ideal of morphisms in the sense of \cite{Ehresmann}, so we can consider kernels and cokernels relatively to the ideal $\mathcal{N}_{\zeros}$ as done in \cite{Ehresmann2}. We will call them $\zeros$-kernels and $\zeros$-cokernels respectively. Let us recall the definition:
	
	\begin{defi}
		Let $\mathcal{C}$ be a category with a fixed class $\zeros$ of zero objects. Let $f \colon A \to B$ be a morphism in $\mathcal{C}$. We say that a morphism $k \colon K \to A$ in $\mathcal{C}$ is
		a \emph{$\zeros$-kernel} of $f$ if the following properties are satisfied:
		\begin{itemize}
			\item[(a)] $fk \in \mathcal{N}_{\zeros}$;
			\item[(b)] whenever $e: E \to A$ is a morphism in $\mathcal{C}$ and $fe \in \mathcal{N}_{\zeros}$, then there exists a unique morphism $\varphi \colon E \to K$ in $\mathcal{C}$ such that $k \varphi = e$.
		\end{itemize}
		$\zeros$-cokernels are defined dually.
	\end{defi}
	
	In a regular category $\mathcal{C}$ with a class $\zeros$ of zero objects, $\zeros$-kernels of morphisms always exist, while the existence of $\zeros$-cokernels is not guaranteed. See \cite{Cappelletti} for a detailed discussion on this. We will frequently denote by $K[f]$ the domain of a $\zeros$-kernel of a morphism $f$.
	
	The notions of $\zeros$-kernels and $\zeros$-cokernels allow the definition of exact sequences relatively to $\zeros$:
	
	\begin{defi}
		Let $\mathcal{C}$ be a category with a fixed class $\zeros$ of zero objects. A \emph{short $\zeros$-exact sequence} (or simply short exact sequence, when there is no ambiguity) in $\mathcal{C}$ is a sequence
		\[\begin{tikzcd}
			A & B & C
			\arrow["f", from=1-1, to=1-2]
			\arrow["g", from=1-2, to=1-3]
		\end{tikzcd}\]
		in which $f$ is a $\zeros$-kernel of $g$ and $g$ is a $\zeros$-cokernel of $f$.
	\end{defi}
	
	In \cite{GrandisJan torsion th} a non-pointed version of the classical notion of torsion theory has been proposed, using kernels and cokernels with respect to a suitable ideal $\mathcal{N}$ of morphisms, under the assumption that $\mathcal{N}$-kernels and $\mathcal{N}$-cokernels of the identity morphisms exist. In \cite{CM torsion theories 1} we considered essentially the same notion, with respect to the ideal $\mathcal{N}_{\zeros}$ of arrows factoring through the class of zero objects in a regular category $\mathcal{C}$, but without the assumption of the existence of kernels and cokernels. Let us recall our definition (that should also be compared with the one of pretorsion theory introduced in \cite{Facchini2}):
	
	\begin{defi} \label{definizione preesatta}
		Let $\mathcal{C}$ be a category with a class $\zeros$ of zero objects. A \emph{$\zeros$-torsion theory} (or simply torsion theory, if there is no ambiguity) $(\tort, \torf)$ in $\mathcal{C}$ consists of two full, replete subcategories $\tort, \torf$ of $\mathcal{C}$ satisfying the following conditions:
		\begin{itemize}
			\item[(a)] $\zeros = \tort \cap \torf$;
			\item[(b)] $\Hom(T, F) \subseteq \zideal$ for every object $T \in \tort$ , $F \in \torf$;
			\item[(c)] for every object $A$ of $\mathcal{C}$ there is a short $\zeros$-exact sequence
			\[\begin{tikzcd}
				{\T(A)} & A & {\F(A)}
				\arrow["{t_A}", from=1-1, to=1-2]
				\arrow["{\eta_A}", from=1-2, to=1-3]
			\end{tikzcd}\]
			with $\T(A) \in \tort$ and $\F(A) \in \torf$.
		\end{itemize}
	\end{defi}
	
	As shown in \cite{Facchini2}, every $\zeros$-torsion theory as above gives rise to functors $\T \colon \mathcal{C} \to \tort$ and $\F \colon \mathcal{C} \to \torf$. The first one is right adjoint to the inclusion $\tort \hookrightarrow \mathcal{C}$, while the second is left adjoint to the inclusion $\torf \hookrightarrow \mathcal{C}$. Hence $\tort$ is a monocoreflective subcategory of $\mathcal{C}$, while $\torf$ is an epireflective subcategory. \\
	
	In \cite{Gran torsione}, the authors studied, in the framework of homological categories (namely pointed, regular, protomodular \cite{Bourn protomod} categories), the connections between the notion of torsion theory and the ones of factorization system and Galois structure. In \cite{CM torsion theories 1} we extended these results to the context of a regular protomodular category $\mathcal{C}$ with a class $\zeros$ of zero objects, and in particular to the case of regular protomodular categories in which the unique morphism $0 \to 1$ from the initial to the terminal object is a regular epimorphism, taking as $\zeros$ the subcategory of regular quotients of $0$. We briefly recall here the main ingredients needed to extend those connections.
	
	\begin{defi}
		Let $\mathcal{C}$ be a category with a class $\zeros$ of zero objects. A torsion theory $\torsione$ in $\mathcal{C}$ satisfies \emph{Condition (N)} if, for every diagram
		\[\begin{tikzcd}
			{\T(K[f])} & {K[f]} & A & B
			\arrow["f", from=1-3, to=1-4]
			\arrow["k", from=1-2, to=1-3]
			\arrow["{t_{K[f]}}", from=1-1, to=1-2]
		\end{tikzcd}\]
		where $k = \Zker(f)$, $kt_{K[f]}$ is the $\zeros$-kernel of some arrow.
	\end{defi}
	
	\begin{defi}
		Let $\mathcal{C}$ be a category with a class $\zeros$ of zero objects. A torsion theory $\torsione$ in $\mathcal{C}$ satisfies \emph{Condition (M)} if, for every short exact sequence
		\[\begin{tikzcd}
			T & A & Q,
			\arrow["k", from=1-1, to=1-2]
			\arrow["q", from=1-2, to=1-3]
		\end{tikzcd}\]
		where $T \in \tort$, $\F(T) \cong \Z(Q)$.
	\end{defi}
	
	Condition (N) above was already considered in \cite{Gran torsione} in the pointed case, where it is equivalent to the fact that all torsion subobjects are characteristic in the sense of \cite{CM char}. Condition (M), instead, is typical of the non-pointed situation, being trivially satisfied in pointed categories.
	
	\begin{prop}[\cite{CM torsion theories 1}, Proposition $3.5$]\label{Da torsione a sistema}
		Let $\mathcal{C}$ be a regular protomodular category with a class $\zeros$ of zero objects. Consider a torsion theory $\torsione$ in $\mathcal{C}$ satisfying Conditions (N) and (M). Then, the classes of arrows $(\facte, \factm)$, where
		\[ \facte \coloneqq \{ e \in \frecce \, | \, e \text{ is a $\zeros$-cokernel, } K[e] \in \tort \} \]
		and
		\[ \factm \coloneqq \{ m \in \frecce | K[m] \in \torf \} \]
		form a factorization system which is stable with respect to the class of arrows inverted by the functor $\Z$.
	\end{prop}
	
	To get a torsion theory from a factorization system, we need an additional condition on our class $\zeros$:
	
	\begin{defi}[\cite{CM torsion theories 1}]\label{def terminal class}
		The class $\zeros$ of zero objects, in a category $\mathcal{C}$, is \emph{terminal} if for every $A \in \mathcal{C}$ there exists an arrow $v_A \colon A \rightarrow \V(A)$ such that:
		\begin{itemize}
			\item[(a)] $ \V(A) \in \zeros$;
			\item[(b)] for every $\chi \colon A \rightarrow Z$, with $Z \in \zeros$, there exists a (necessarily unique) arrow $\varphi \colon Z \rightarrow \V(A)$ such that $\varphi \chi = v_A$.
		\end{itemize}
	\end{defi}
	
	It was observed in \cite{CM torsion theories 1} that the assignment $A \mapsto \V(A)$ actually gives rise to a functor $\V \colon \op{\mathcal{C}} \rightarrow \zeros$, whose image on arrows is the following: given $f \colon A \to B$ in $\mathcal{C}$, $\V(f)$ is the unique arrow such that $\V(f)v_Bf=v_A$.
	
	In our main framework of a regular protomodular category in which the unique morphism $0 \to 1$ from the initial to the terminal object is a regular epimorphism, the class $\zeros$ formed by regular quotients of $0$ is obviously terminal, since the terminal object $1$ belongs to $\zeros$.
	
	\begin{prop}[\cite{CM torsion theories 1}, Proposition $3.10$]
		Let $\mathcal{C}$ be a regular category with a terminal class $\zeros$ of zero objects. Consider a factorization system $(\facte, \factm)$ which is stable with respect to the class of arrows inverted by $\Z$, and where the arrows in $\facte$ are $\zeros$-cokernels. Then, the pair of subcategories $\torsione$ given by
		\[ \tort \coloneqq \{ T \in \mathcal{C} \,|\, \exists \, T \rightarrow Z \text{ belonging to } \facte \text{ (where $Z \in \zeros$)}\} \]
		and
		\[ \torf \coloneqq \{ F \in \mathcal{C} \,|\, F \rightarrow \V(F) \text{ belongs to } \factm\} \]
		forms a torsion theory such that $\tort \cap \torf = \zeros$ and satisfying (N) and (M).
	\end{prop}
	
	In order to talk about the connection between torsion theories and Galois structure, we need to recall from \cite{janelidze1990pure} the main notions of categorical Galois theory:
	
	\begin{defi}
		Let $\mathcal{C}$ be a category with pullbacks. A class $\mathbb{C}$ of arrows in $\mathcal{C}$ is admissible when:
		\begin{itemize}
			\item[(a)] every isomorphism is in $\mathbb{C}$;
			\item[(b)] $\mathbb{C}$ is closed under composition;
			\item[(c)] in the pullback
			\[\begin{tikzcd}
				\bullet & \bullet \\
				\bullet & \bullet
				\arrow["c", from=1-1, to=1-2]
				\arrow["a"', from=2-1, to=2-2]
				\arrow["d"', from=1-1, to=2-1]
				\arrow["b", from=1-2, to=2-2]
				\arrow["\lrcorner"{anchor=center, pos=0}, draw=none, from=1-1, to=2-2]
			\end{tikzcd}\]
			if $a$ and $b$ are in $\mathbb{C}$ then $c$ and $d$ are in $\mathbb{C}$, too.
		\end{itemize}
	\end{defi}
	
	Let $\mathbb{C}$ be an admissible class of morphisms in a category $\mathcal{C}$. For an object $C$ of $\mathcal{C}$, we write $\mathbb{C}/C$ for the category whose objects are the pairs $(X,f)$ where $f \colon X \rightarrow C$ is in $\mathbb{C}$, and whose arrows $h \colon (X,f) \rightarrow (Y,g)$ are all arrows in $\mathcal{C}$ such that $gh=f$.
	
	\begin{defi}
		A \emph{relatively admissible adjunction} consists of an adjunction $\S \dashv \C \colon \mathcal{P} \rightarrow \mathcal{A}$ (where $\mathcal{A}$ and $\mathcal{P}$ have pullbacks) and two admissible classes $\mathbb{A} \subseteq \mathcal{A}$, $\mathbb{P} \subseteq \mathcal{P}$ of arrows, such that
		\begin{itemize}
			\item[(a)] $\S(\mathbb{A}) \subseteq \mathbb{P}$,
			\item[(b)] $\C(\mathbb{P}) \subseteq \mathbb{A}$,
			\item[(c)] for every object $A$ of $\mathcal{A}$ the component $\eta_A$ of the unit of the adjunction $\S \dashv \C$ is in $\mathbb{A}$,
			\item[(d)] for every object $X$ of $\mathcal{P}$ the component $\varepsilon_X$ of the counit of the adjunction $\S \dashv \C$ is in $\mathbb{P}$.
		\end{itemize}
		We denote a relatively admissible adjunction with $(\S,\C,\mathbb{A},\mathbb{P})$.
	\end{defi}
	
	Given a relatively admissible adjunction $\S \dashv \C \colon \mathcal{P} \rightarrow \mathcal{A}$, we denote by \linebreak $\C_A \colon \mathbb{P}/\S(A) \rightarrow \mathbb{A}/A$, where $A$ is an object of $\mathcal{A}$, the functor operating on objects as $\C_A(g \colon P \rightarrow \S(A)) \coloneqq \pi_A$ where $\pi_A$ is the upper morphism in the following pullback diagram:
	\[\begin{tikzcd}
		K & A \\
		{\C(P)} & {\C\S(A),}
		\arrow["{\C(g)}"', from=2-1, to=2-2]
		\arrow["{\eta_A}", from=1-2, to=2-2]
		\arrow["{\pi_A}", from=1-1, to=1-2]
		\arrow["{\pi_{\C(P)}}"', from=1-1, to=2-1]
		\arrow["\lrcorner"{anchor=center, pos=0}, draw=none, from=1-1, to=2-2]
	\end{tikzcd}\]
	with $\eta_A$ the component in $A$ of the unit of the adjunction $\S \dashv \C$.
	
	\begin{defi}
		A relatively admissible adjunction $(\S,\C,\mathbb{A},\mathbb{P})$ is \emph{admissible} when the functor $\C_A$ is full and faithful for every object $A$ of $\mathcal{A}$.
	\end{defi}
	
	\begin{defi}
		Given an admissible adjunction $(\S,\C,\mathbb{A},\mathbb{P})$,
		\begin{itemize}
			\item[(a)] a \emph{trivial extension} is an arrow $f \colon A \rightarrow B$ of $\mathbb{A}$ such that the square
			\[\begin{tikzcd}
				A & {\C\S(A)} \\
				B & {\C\S(B)}
				\arrow["f"', from=1-1, to=2-1]
				\arrow["{\C\S(f)}", from=1-2, to=2-2]
				\arrow["{\eta_A}", from=1-1, to=1-2]
				\arrow["{\eta_B}"', from=2-1, to=2-2]
				\arrow["\lrcorner"{anchor=center, pos=0}, draw=none, from=1-1, to=2-2]
			\end{tikzcd}\]
			is a pullback (where $\eta$ is the unit of the adjunction $\S \dashv \C$);
			\item[(b)] a \emph{normal extension} is an arrow $f$ of $\mathbb{A}$ such that it is an effective descent morphism relatively to $\mathbb{A}$ and its kernel pair projections are trivial extensions;
			\item[(c)] a \emph{central extension} is an arrow $f$ of $\mathbb{A}$ such that there exists an effective descent morphism $g$ relatively to $\mathbb{A}$ such that the pullback $g^*(f)$ of $f$ along $g$ is a trivial extension.
		\end{itemize}
	\end{defi}
	
	To extend to our non-pointed context the connection between torsion theories and admissible Galois structures developed in \cite{Gran torsione} for homological categories, we need our torsion theories to satisfy some additional conditions, that are automatically satisfied in the pointed case:
	
	\begin{defi}
		A torsion theory $\torsione$, in a category $\mathcal{C}$ with a class $\zeros$ of zero objects, satisfies \emph{Condition (M')} if, for every short exact sequence
		\[\begin{tikzcd}
			\T(A) & A & \F(A),
			\arrow["t_A", from=1-1, to=1-2]
			\arrow["\eta_A", from=1-2, to=1-3]
		\end{tikzcd}\]
		one has $\F(\T(A))=\Z(\F(A))$.
	\end{defi}
	
	It is immediate to see that Condition (M) implies Condition (M').
	
	\begin{defi}
		Let $\mathcal{C}$ be a category with a class $\zeros$ of zero objects, and $\torsione$ a torsion theory in $\mathcal{C}$. Consider the following pullback in $\mathcal{C}$
		\[\begin{tikzcd}
			{T'} & {Z'} \\
			T & Z,
			\arrow["\chi", from=1-1, to=1-2]
			\arrow[from=1-1, to=2-1]
			\arrow["\lrcorner"{anchor=center, pos=0}, draw=none, from=1-1, to=2-2]
			\arrow[from=1-2, to=2-2]
			\arrow["{\eta_T}"', from=2-1, to=2-2]
		\end{tikzcd}\]
		where $T \in \tort$, $Z,Z' \in \zeros$, and $\eta_T$ is the $T$-component of the unit of the reflection induced by $\torsione$. We say that $\torsione$ satisfies \emph{Condition (S)} if, for every pullback diagram of the type above, one has $T' \in \tort$ and $\chi=\eta_{T'}$.
	\end{defi}
	
	\begin{prop}[\cite{CM torsion theories 1}, Corollary $4.16$]
		Let $\mathcal{C}$ be a regular protomodular category with a fixed class $\zeros$ of zero objects. Consider a torsion theory $\torsione$ in $\mathcal{C}$ satisfying Conditions (M') and (S). Then, the data $\Gamma \coloneqq (\F, \i, \Arr(\mathcal{C}), \Arr(\torf))$, where $\i \colon \torf \hookrightarrow \mathcal{C}$ is the inclusion functor, determine an admissible Galois structure.
	\end{prop}
	
	To get a nice description of central and normal extensions for a Galois structure arising from a torsion theory, even in the pointed case studied in \cite{Gran torsione}, a property of the reflector $\F$, which is a generalization of additivity for a functor between additive categories, is needed:
	
	\begin{defi} \label{def protoadditive}
		Let $\mathcal{C}$ be a protomodular category with pullbacks and a fixed class $\zeros$ of zero objects. A functor $\F \colon \mathcal{C} \rightarrow \torf$, where $\torf$ is a protomodular category, is said to be \emph{$\mathcal{Z}$-protoadditive} (or simply protoadditive if there is no ambiguity) if $\F$ preserves the pullbacks of split epimorphisms along arrows inverted by $\Z$. We say that a torsion theory $\torsione$ in $\mathcal{C}$ satisfies Condition \emph{(P)} if the reflector $\F$ is $\mathcal{Z}$-protoadditive.
	\end{defi}
	
	\begin{teo}[\cite{CM torsion theories 1}, Theorem $4.17$] \label{centrali e torsione}
		Let $\mathcal{C}$ be a Barr-exact protomodular category with a class $\zeros$ of zero objects. Consider a torsion theory $\torsione$ in $\mathcal{C}$ satisfying Conditions (S), (M'), and (P). Let $f$ be an effective descent morphism in $\mathcal{C}$ (i.e.\ a regular epimorphism, since $\mathcal{C}$ is Barr-exact) and let $\Gamma$ be the Galois structure associated with the reflector $\F$. Then, the following conditions are equivalent:
		\begin{itemize}
			\item[(a)] $f$ is a normal extension for $\Gamma$;
			\item[(b)] $f$ is a central extension for $\Gamma$;
			\item[(c)] $K[f] \in \torf$.
		\end{itemize}
	\end{teo}
	
	In \cite{CM torsion theories 1} several examples of non-pointed torsion theories, satisfying all the conditions we mentioned, are described.
	
	\section{Torsion theories in the category of arrows}
	It was shown in \cite{Gran torsione} that a torsion theory $\torsione$ satisfying Condition (N) in a homological category $\mathcal{C}$ induces a torsion theory $(\tort_1, \torf_1)$ in the category $\mathrm{Arr}(\mathcal{C})$ of arrows of $\mathcal{C}$, satisfying itself Condition (N). Iterating the process, one obtains a sequence $(\tort_n, \torf_n)$ of torsion theories in $\mathrm{Arr}^n(\mathcal{C})$ satisfying Condition (N). If, moreover, $\torsione$ satisfies Condition (P), then all the iterated torsion theories $(\tort_n, \torf_n)$ satisfy (P), as well. \\
	
	In this section we will show that, also in our non-pointed context, the torsion theory $(\tort_1, \torf_1)$ in $\mathrm{Arr}(\mathcal{C})$ can be constructed. First of all, we need to identify the class of zero objects in $\mathrm{Arr}(\mathcal{C})$:
	
	\begin{prop}
		If $\zeros \subseteq \mathcal{C}$ is a class of zero objects in a regular category $\mathcal{C}$, then $\arrz \subseteq \arrc$, seen as a full subcategory, is a class of zero objects.
	\end{prop}
	
	\begin{proof}
		Let $z_1, z_2 \in \arrz$. A morphism in $\arrz$ between these two objects consists simply of a morphism in $\mathcal{C}$ between their domains and a morphism in $\mathcal{C}$ between their codomains. Since $\zeros$ is a posetal category, it follows immediately that $\arrz$ inherits the property of being posetal.
		
		Now, let $a \colon A \to B$ be an object of $\arrc$. Since $\zeros$ is monocoreflective in $\mathcal{C}$, the coreflector $\Z$ provides the following commutative diagram, where the horizontal arrows (the components of the counit) are monomorphisms:
		\[\begin{tikzcd}
			{\Z(A)} & A \\
			{\Z(B)} & B.
			\arrow["{\varepsilon_A}", tail, from=1-1, to=1-2]
			\arrow["{\Z(a)}"', from=1-1, to=2-1]
			\arrow["a", from=1-2, to=2-2]
			\arrow["{\varepsilon_B}"', tail, from=2-1, to=2-2]
		\end{tikzcd}\]
		Hence, the pair $(\varepsilon_A, \varepsilon_B)$, viewed as a morphism in $\arrc$, is a monomorphism.
		
		Finally, consider an object $z \colon Z_1 \to Z_2$ in $\arrz$ and a morphism $(x_1, x_2) \colon z \to a$ in $\arrc$. Since $\zeros$ is coreflective in $\mathcal{C}$, the universal property of the counit guarantees the existence of a unique morphism $y_1 \colon Z_1 \to \Z(A)$ such that $\varepsilon_A y_1 = x_1$, and a unique morphism $y_2 \colon Z_2 \to \Z(B)$ such that $\varepsilon_B y_2 = x_2$. Since $Z_1$ and $\Z(B)$ belong to $\zeros$, and $\zeros$ is a posetal category, there is at most one morphism from $Z_1$ to $\Z(B)$. Thus, the composites $\Z(a) y_1$ and $y_2 z$ must coincide, which means that the left square of the following diagram commutes:
		\[\begin{tikzcd}
			& {Z_1} & \\
			{\Z(A)} && A \\
			& {Z_2} \\
			{\Z(B)} && B.
			\arrow["{y_1}"', from=1-2, to=2-1]
			\arrow["{x_1}", from=1-2, to=2-3]
			\arrow["z"{pos=0.7}, from=1-2, to=3-2]
			\arrow["{\varepsilon_A}"{pos=0.3}, from=2-1, to=2-3]
			\arrow["{\Z(a)}"', from=2-1, to=4-1]
			\arrow["a", from=2-3, to=4-3]
			\arrow["{y_2}"', from=3-2, to=4-1]
			\arrow["{x_2}", from=3-2, to=4-3]
			\arrow["{\varepsilon_B}"'{pos=0.3}, from=4-1, to=4-3]
		\end{tikzcd}\]
		Consequently, $(y_1, y_2)$ is the unique morphism in $\arrc$ such that $(\varepsilon_A, \varepsilon_B)(y_1, y_2) = (x_1, x_2)$. This proves that $\arrz$ is coreflective in $\arrc$, with the $a$-component of the counit given by $\varepsilon_a = (\varepsilon_A, \varepsilon_B)$.
		
		Finally, since the coreflector $\Z \colon \mathcal{C} \to \zeros$ inverts monomorphisms, the monomorphisms in $\arrc$ are levelwise monomorphisms, and the coreflector $\Z_1 \colon \arrc \to \arrz$ acts levelwise, it is clear that $\Z_1$ inverts monomorphisms, too.
	\end{proof}
	
	Moreover:
	
	\begin{prop}
		If $\zeros$ is terminal (in the sense of Definition \ref{def terminal class}), then $\arrz$ is terminal, too.
	\end{prop}
	
	\begin{proof}
		Given an object $f \colon A \to B$ in $\arrc$, consider the diagram
		\[\begin{tikzcd}
			A && {\V(A)} \\
			& {Z_1} \\
			B && {\V(A)} \\
			& {Z_2}
			\arrow["{v_A}", from=1-1, to=1-3]
			\arrow["{\chi_1}"', from=1-1, to=2-2]
			\arrow["f"', from=1-1, to=3-1]
			\arrow[equal, from=1-3, to=3-3]
			\arrow["\varphi"', from=2-2, to=1-3]
			\arrow["z"'{pos=0.2}, from=2-2, to=4-2]
			\arrow["{\V(f) v_B}"'{pos=0.3}, from=3-1, to=3-3]
			\arrow["{\chi_2}"', from=3-1, to=4-2]
			\arrow["\psi"', from=4-2, to=3-3]
		\end{tikzcd}\]
		where $(\chi_1, \chi_2) \colon f \to z$ with $z \in \arrz$, $\varphi$ is the unique morphism in $\mathcal{C}$ such that $\varphi \chi_1 = v_A$, and $\psi$ is the unique morphism in $\mathcal{C}$ such that $\psi \chi_2 f = v_A$ (both $\varphi$ and $\psi$ are induced by the universal property of $v_A$). Clearly $(\varphi, \psi)$ is a morphism in $\arrc$. We define $\Hv(f) \coloneqq 1_{\V(A)}$. In order to show that this assignment makes $\arrz$ a terminal class of zero objects, we must prove that the bottom triangle of the diagram above commutes, i.e., that $\psi \chi_2 = \V(f) v_B$. By the universal property of $v_B$, there exists a unique $\theta$ such that $\theta \chi_2 = v_B$. Thus, we obtain the following commutative diagram (the right part commutes since there is at most one arrow between zero objects):
		\[\begin{tikzcd}
			B && {Z_2} & {\V(A)} \\
			& {\V(B)} \\
			& {\V(A).}
			\arrow["{\chi_2}", from=1-1, to=1-3]
			\arrow["{v_B}"', from=1-1, to=2-2]
			\arrow["\psi", from=1-3, to=1-4]
			\arrow["\theta", from=1-3, to=2-2]
			\arrow["{\V(f)}"', from=2-2, to=3-2]
			\arrow[equal, from=3-2, to=1-4]
		\end{tikzcd}\]
		Consequently, $\psi \chi_2 = \V(f) \theta \chi_2 = \V(f) v_B$.
	\end{proof}
	
	Now we are ready to state the main result of this section:
	
	\begin{prop}
		Let $\mathcal{C}$ be a regular protomodular category with a class $\zeros$ of zero objects. Consider a torsion theory $\torsione$ in $\mathcal{C}$ satisfying Conditions (N) and (M), and the corresponding factorization system $(\ce, \cm)$. Then the pair of full subcategories $(\tort_1, \torf_1)$, where
		\[ \torf_1 = \{ m \in \arrc \,|\, K[m] \in \torf \} \]
		and
		\[ \tort_1 = \{ g \in \arrc \,|\, g \colon T \to Z, T \in \tort, Z \in \zeros \} \]
		is an $\arrz$-torsion theory in the category $\mathrm{Arr}(\mathcal{C})$.
	\end{prop}
	
	\begin{proof}
		Let us first prove that $\tort_1 \cap \torf_1 = \arrz$. The inclusion $\arrz \subseteq \tort_1 \cap \torf_1$ follows immediately from the definition of $\tort_1$ and $\torf_1$. For the other inclusion, take $g \colon T \to Z$ in $\tort_1 \cap \torf_1$ and consider the $(\ce,\cm)$-factorization of $g$:
		\[\begin{tikzcd}
			T && Z \\
			& X.
			\arrow["{g \in \cm}", from=1-1, to=1-3]
			\arrow["{e \in \ce}"', from=1-1, to=2-2]
			\arrow["{m \in \cm}"', from=2-2, to=1-3]
		\end{tikzcd}\]
		Since $m \in \cm$, we get $X \cong K[m] \in \torf$. Moreover, due to the essential uniqueness of the factorization, $e$ is an isomorphism. So $T \cong X \in \torf \cap \tort = \zeros$, and hence $g \in \arrz$.
		
		Consider now an arrow $(a,b) \colon t \to m$, with $t \colon T \to Z$ in $\tort_1$ and $m \colon A \to B$ in $\torf_1$. We take the $(\ce,\cm)$-factorization $t=ze$ of $t$, where $e \colon T \to Z'$ and $z \colon Z' \to Z$. The fact that the codomain of $e$ is a zero object follows from the way the $(\ce,\cm)$-factorization is obtained (see the proof of Proposition $3.5$ in \cite{CM torsion theories 1}). Observe that in the commutative square
		\[\begin{tikzcd}
			T & {Z'} \\
			A & B
			\arrow["e", from=1-1, to=1-2]
			\arrow["a"', from=1-1, to=2-1]
			\arrow["d"', from=1-2, to=2-1]
			\arrow["bz", from=1-2, to=2-2]
			\arrow["m"', from=2-1, to=2-2]
		\end{tikzcd}\]
		the existence of the diagonal $d$ is guaranteed by the fact that $(\ce,\cm)$ is a factorization system. Therefore, $(a,b)$ factors through $\arrz$ as $(a,b) \colon t \xrightarrow{(e,1_Z)} z \xrightarrow{(d,b)} m$.
		
		Finally, let us show that, for any $f \in \arrc$, there exists a short $\arrz$-exact sequence whose middle object is $f$, whose $\arrz$-kernel is in $\tort_1$ and whose $\arrz$-cokernel is in $\torf_1$. As shown in \cite[Proposition $3.5$]{CM torsion theories 1}, the $(\ce,\cm)$-factorization of $f$ is given by:
		\[\begin{tikzcd}
			{\T(K)} & K & A && B \\
			&& {\F(K)} & C
			\arrow["{t_K}", tail, from=1-1, to=1-2]
			\arrow["k", tail, from=1-2, to=1-3]
			\arrow["{\eta_K}"', from=1-2, to=2-3]
			\arrow["f", from=1-3, to=1-5]
			\arrow["e"', from=1-3, to=2-4]
			\arrow["{k'}"', tail, from=2-3, to=2-4]
			\arrow["m"', from=2-4, to=1-5]
		\end{tikzcd}\]
		where $e=\Zcoker(kt_K)$, $kt_K=\Zker(e)$, and $k'=\Zker(m)$. Observe that, since $K[e]=\T(K) \in \tort$, the unique arrow from $\T(K) \to \Z(B)$ is $\Z(m)t$, where $t$ is defined by the pullback
		\[\begin{tikzcd}
			{K[e]} & {\Z(C)} \\
			A & C.
			\arrow["t", from=1-1, to=1-2]
			\arrow["{kt_K}"', tail, from=1-1, to=2-1]
			\arrow["\lrcorner"{anchor=center, pos=0.125}, draw=none, from=1-1, to=2-2]
			\arrow["{\varepsilon_C}", tail, from=1-2, to=2-2]
			\arrow["e"', from=2-1, to=2-2]
		\end{tikzcd}\]
		So, $fkt_K=\varepsilon_B\Z(m)t$. Hence, the following diagram is commutative:
		\[\begin{tikzcd}
			{K[e]} & A & C \\
			{\Z(C)} \\
			{\Z(B)} & B & B.
			\arrow["{kt_K}", tail, from=1-1, to=1-2]
			\arrow["t"', from=1-1, to=2-1]
			\arrow["e", from=1-2, to=1-3]
			\arrow["f", from=1-2, to=3-2]
			\arrow["m", from=1-3, to=3-3]
			\arrow["{\Z(m)}"', from=2-1, to=3-1]
			\arrow["{\varepsilon_B}"', tail, from=3-1, to=3-2]
			\arrow[equal, from=3-2, to=3-3]
		\end{tikzcd}\]
		Moreover, we have that $\Z(m)t \in \tort_1$ and $m \in \torf_1$. Since, in $\arrc$, $\arrz$-kernels and $\arrz$-cokernels are computed levelwise as in $\mathcal{C}$, we get that $(kt_K,\varepsilon_B)=\ArrZker(e,1_B)$ and $(e,1_B)=\ArrZcoker(kt_K, \varepsilon_B)$.
	\end{proof}
	
	We will denote by $\Ht \colon \arrc \to \tort_1$ and $\Hf \colon \arrc \to \torf_1$ the coreflector to the torsion part and the reflector to the torsion-free part of $(\tort_1, \torf_1)$, respectively.
	
	It is easy to verify that Conditions (M) and (S) pass from $\torsione$ to $(\tort_1, \torf_1)$:
	
	\begin{prop}
		If $\torsione$ satisfies Conditions (N) and (M), then $(\tort_1, \torf_1)$ satisfies Condition (M), too.
	\end{prop}
	
	\begin{proof}
		By assumption, we know that, if we consider in $\mathcal{C}$ a short $\zeros$-exact sequence $T \xrightarrow{t} A \xrightarrow{q} Q$, with $T \in \tort$, then $\F(T) = \Z(Q)$. Now, consider an exact sequence in $\arrc$
		\[\begin{tikzcd}
			T & A & {Q_1} \\
			Z & B & {Q_2,}
			\arrow["{k_1}", tail, from=1-1, to=1-2]
			\arrow["t"', from=1-1, to=2-1]
			\arrow["{q_1}", from=1-2, to=1-3]
			\arrow["a"', from=1-2, to=2-2]
			\arrow["p", from=1-3, to=2-3]
			\arrow["{k_2}"', tail, from=2-1, to=2-2]
			\arrow["{q_2}"', from=2-2, to=2-3]
		\end{tikzcd}\]
		with $t \in \tort_1$. Observe that $\Z(Q_1) = \F(T)$ and $\Z(Q_2)=Z$, since the two rows are short $\zeros$-exact sequences in $\mathcal{C}$. We compute $\Hf(t)$: consider the commutative triangle
		\[\begin{tikzcd}
			T && Z \\
			& {Z'}
			\arrow["t", from=1-1, to=1-3]
			\arrow["{e \in \ce}"', from=1-1, to=2-2]
			\arrow["{m \in \cm}"', from=2-2, to=1-3]
		\end{tikzcd}\]
		given by the $(\ce,\cm)$-factorization of $t$. One has $\Hf(t)=m$. Moreover, $Z'=\Zcoker(1_T)=\F(T)$. Hence, one has $\Hf(t)=\Z(p)$.
	\end{proof}
	
	\begin{prop}
		If $\torsione$ satisfies Conditions (N), (M), and (S), then $(\tort_1, \torf_1)$ satisfies Condition (S), too.
	\end{prop}
	
	\begin{proof}
		Let $f \colon T \to Z$ be an object of $\tort_1$. Consider the $(\facte, \factm)$-factorization of $f$, whose existence is guaranteed by Proposition \ref{Da torsione a sistema}, since $\torsione$ satisfies Conditions (N) and (M):
		\[ \begin{tikzcd}
			T && Z \\
			& {\F(T),}
			\arrow["f", from=1-1, to=1-3]
			\arrow["{\eta_T \in \ce}"', from=1-1, to=2-2]
			\arrow["{z \in \cm}"', from=2-2, to=1-3]
		\end{tikzcd}\]
		where $\F(T) \in\zeros$. Observe that the $f$-component of the reflection determined by $\Hf$ is $\eta_f = (\eta_T, 1_Z)$. Let $\overline{z} \colon Z_2 \to Z_1$ be an object of $\arrz$, and consider a morphism $(z_2,z_1) \colon \overline{z} \to z$ in $\arrc$. Computing the pullback of $(\eta_T, 1_Z)$ along $(z_2,z_1)$ yields the following commutative cube:
		\[\begin{tikzcd}
			& {T'} && {Z_2} \\
			{Z_1} && {Z_1} \\
			& T && {\F(T)} \\
			Z && Z,
			\arrow["\chi"{pos=0.3}, from=1-2, to=1-4]
			\arrow["{f'}"', from=1-2, to=2-1]
			\arrow["g"{pos=0.2}, from=1-2, to=3-2]
			\arrow["\lrcorner"{anchor=center, pos=0.125}, draw=none, from=1-2, to=3-4]
			\arrow["{\overline{z}}", from=1-4, to=2-3]
			\arrow["{z_2}"{pos=0.2}, from=1-4, to=3-4]
			\arrow[equal, from=2-1, to=2-3]
			\arrow["{z_1}"'{pos=0.2}, from=2-1, to=4-1]
			\arrow["\lrcorner"{anchor=center, pos=0.125}, draw=none, from=2-1, to=4-3]
			\arrow["{z_1}"{pos=0.2}, from=2-3, to=4-3]
			\arrow["{\eta_T}"{pos=0.3}, from=3-2, to=3-4]
			\arrow["f"', from=3-2, to=4-1]
			\arrow["z", from=3-4, to=4-3]
			\arrow[equal, from=4-1, to=4-3]
		\end{tikzcd}\]
		where the front and back squares are pullback diagrams. Since $\torsione$ satisfies Condition (S), we obtain $T' \in \tort$ and $\chi = \eta_{T'}$. This implies the validity of Condition (S) in $(\tort_1, \torf_1)$.
	\end{proof}
	
	The fact that Conditions (N) and (P) pass from $\torsione$ to $(\tort_1, \torf_1)$ is much trickier. We will address this problem in the next sections.

	\section{Condition (N) for the torsion theory in the category of arrows}
	
	In the pointed case, it was proved in \cite{Gran torsione} that, if a torsion theory $\torsione$ in a homological category $\mathcal{C}$ satisfies Condition (N), so does $(\tort_1, \torf_1)$ in $\arrc$. The proof of this fact, presented in \cite{Gran torsione}, relies on the fact that, given two morphisms $f \colon A \to B$ and $g \colon A \to C$ in $\mathcal{C}$, the intersection of the kernels of $f$ and $g$ is isomorphic to the kernel of the morphism $\langle f, g \rangle \colon A \to B \times C$ induced by the universal property of the product. This fact, however, is false for $\zeros$-kernels in our context. In fact, the intersection of $\zeros$-kernels is not always a $\zeros$-kernel, and even when it is, it is not always the $\zeros$-kernel of the induced morphism. Let us show two counterexamples, the first in the category $\mathbf{Boole}$ of Boolean algebras, the second in the category $\mathbf{Ring}$ of unitary rings.
	
	\begin{example}
		In $\boole$ the $\zeros$-kernel of $\spigolose{f,g}$, in general, is different from the intersection of the $\zeros$-kernel of $f$ and the $\zeros$-kernel of $g$. Indeed, consider the four-element Boolean algebra $A \coloneq \graffe{0,x, \lnot x, 1}$ and let $f \colon A \to \Two$, $g \colon A \to \Two$ (where $\Two$ is the two-element Boolean algebra) defined by $f(x) \coloneq 0, g(x) \coloneq 1$. Let $K[\spigolose{f,g}], K[f], K[g]$ be the domains of the $\zeros$-kernel of $\spigolose{f,g}$, $f$, $g$, respectively. Observe that $K[\spigolose{f,g}]= \graffe{y \in A \,|\, (f(y),g(y))=(0,0) \text{ or } (f(y),g(y))=(1,1)} \cong \Two$, while $K[f] \cong A$ and $K[g] \cong A$. Clearly $K[\spigolose{f,g}] \neq K[f] \cap K[g]$.
	\end{example}
	
	\begin{example}
		In $\mathbf{Ring}$ the intersection of $\zeros$-kernels, in general, is not a $\zeros$-kernel. Indeed, in $\mathbf{Ring}$ the $\zeros$-kernels are precisely the subalgebras generated by the ideals. As a concrete example, consider the ring $\mathbb{Z}_2$ of integers modulo $2$. It is easy to check that, in a unitary ring of characteristic $2$, the unitary subring generated by an ideal $I$ is simply $I \cup (I+1)$. Consider, then, the quotient morphisms $\Ze_2[x] \to \Ze_2[x]/(x^2)$ and $\Ze_2[x] \to \Ze_2[x]/(x^2-1)$. Their $\zeros$-kernels are $I \cup (I+1)$ and $J \cup (J+1)$, respectively, where $I=(x^2)$ and $J=(x^2-1)$. Suppose there exists an ideal $K = (g(x)) \subseteq \Ze_2[x]$ such that:
		\[ (I \cup (I+1)) \cap (J \cup (J+1)) = K \cup (K+1). \]
		Noting that $x^2 = 1 \cdot (x^2-1) + 1$, it follows that $x^2 \in I \cap (J+1)$. Consequently, $x^2$ must belong to $K \cup (K+1)$, which implies that either $x^2 \in K$ or $x^2-1 \in K$. We analyze these two cases separately:
		\begin{enumerate}
			\item If $x^2 \in K$, then $g(x)$ must divide $x^2$ in $\Ze_2[x]$. This leads to three possibilities for $g(x)$: $x^2$, $x$, or $1$.
			\begin{itemize}
				\item If $g(x) = x^2$, then $K = I$. Then we would get $I \cup (I+1) \subseteq J \cup (J+1)$. Since $x^3 \in I$, we must have $x^3 \in J$ or $x^3 \in J+1$.
				\begin{itemize}
					\item If $x^3 \in J$, then $x^2-1 \mid x^3$, which is false.
					\item If $x^3 \in J+1$, then $x^2-1 \mid x^3-1$, which is also false, because the remainder of the division of $x^3-1$ by $x^2-1$ is $x-1 \neq 0$.
				\end{itemize}
				\item If $g(x) = x$, then $x \in K \cup (K+1)$, so $x \in I \cup (I+1)$. This is false as $x \notin (x^2)$ and $x \notin (x^2)+1$.
				\item If $g(x) = 1$, then $K \cup (K+1) = \Ze_2[x]$. But this would imply $I \cup (I+1) = \Ze_2[x]$, which is false. Indeed, $I \cup (I+1)$ is the subring of $\Ze_2[x]$ formed by the polynomials with no term of degree $1$.
			\end{itemize}
			
			\item If $x^2-1 \in K$, then $g(x)$ must divide $x^2-1$. In $\Ze_2[x]$, we have $x^2-1 = (x-1)^2$, so $g(x)$ is $x^2-1$, $x-1$, or $1$.
			\begin{itemize}
				\item If $g(x) = x^2-1$, then $K = J$, implying $J \subseteq I \cup (I+1)$. Since $x^3-x = x(x^2-1) \in J$, then $x^3-x \in I$ or $x^3-x \in I+1$.
				\begin{itemize}
					\item If $x^3-x \in I$, then $x^2 \mid x^3-x$, which is false.
					\item If $x^3-x \in I+1$, then $x^2 \mid x^3-x-1$, which is also false.
				\end{itemize}
				\item If $g(x) = x-1$, then $x-1 \in K \cup (K+1) \subseteq I \cup (I+1)$. However, $x-1$ is neither in $(x^2)$ nor in $(x^2)+1$.
				\item If $g(x) = 1$, we get a contradiction as shown in Case $1$.
			\end{itemize}
		\end{enumerate}
		Since all possibilities lead to a contradiction, no such ideal $K$ exists.
	\end{example}
	
	However, what remains true in our context is the following:
	
	\begin{prop} \label{prop torsion of inters of ker}
		Let $\torsione$ be a torsion theory in a regular category $\mathcal{C}$ with a class $\zeros$ of zero objects, and let $f \colon A \to B$ and $g \colon A \to C$ be arrows in $\mathcal{C}$. Then there is an isomorphism
		\[ \T(K[f] \cap K[g]) \cong \T(K[\langle f,g \rangle]). \]
	\end{prop}
	
	\begin{proof}
		Consider the pullback diagrams:
		\begin{equation*}
			\begin{tikzcd}
				{K_1} & {Z_1} & {K_2} & {Z_2} \\
				A & B & A & C \\
				{K_1 \cap K_2} & {K_2} & K & Z \\
				{K_1} & A & A & {B \times C}
				\arrow["{\chi_1}", from=1-1, to=1-2]
				\arrow["{k_1}"', tail, from=1-1, to=2-1]
				\arrow["\lrcorner"{anchor=center, pos=0.125}, draw=none, from=1-1, to=2-2]
				\arrow["{\varepsilon_B}", tail, from=1-2, to=2-2]
				\arrow["{\chi_2}", from=1-3, to=1-4]
				\arrow["{k_2}"', tail, from=1-3, to=2-3]
				\arrow["\lrcorner"{anchor=center, pos=0.125}, draw=none, from=1-3, to=2-4]
				\arrow["{\varepsilon_C}", tail, from=1-4, to=2-4]
				\arrow["f"', from=2-1, to=2-2]
				\arrow["g"', from=2-3, to=2-4]
				\arrow["{i_2}", tail, from=3-1, to=3-2]
				\arrow["{i_1}"', tail, from=3-1, to=4-1]
				\arrow["\lrcorner"{anchor=center, pos=0.125}, draw=none, from=3-1, to=4-2]
				\arrow["{k_2}", tail, from=3-2, to=4-2]
				\arrow["{\chi}", from=3-3, to=3-4]
				\arrow["k"', tail, from=3-3, to=4-3]
				\arrow["\lrcorner"{anchor=center, pos=0.125}, draw=none, from=3-3, to=4-4]
				\arrow["{\varepsilon_{B \times C}}", tail, from=3-4, to=4-4]
				\arrow["{k_1}"', tail, from=4-1, to=4-2]
				\arrow["{\spigolose{f,g}}"', from=4-3, to=4-4]
			\end{tikzcd}
		\end{equation*}
		and let $\pi_B \colon B \times C \to B$ and $\pi_C \colon B \times C \to C$ be the product projections. Since $\varepsilon$ is a natural transformation, we get $\varepsilon_B \Z(\pi_B) = \pi_B \varepsilon_{B \times C}$ and $\varepsilon_C \Z(\pi_C) = \pi_C \varepsilon_{B \times C}$. Observe that $\varepsilon_B \Z(\pi_B) \chi = \pi_B \varepsilon_{B \times C} \chi = \pi_B \spigolose{f,g} k = f k$; so, there exists a unique $\varphi_1 \colon K \to K_1$ such that $k_1 \varphi_1 = k$ and $\chi_1 \varphi_1 = \Z(\pi_B) \chi$. Similarly, there exists a unique $\varphi_2 \colon K \to K_2$ such that $k_2 \varphi_2 = k$ and $\chi_2 \varphi_2 = \Z(\pi_C) \chi$:
		\[\begin{tikzcd}
			K &&& K && \\
			& {K_1} & {Z_1} && {K_2} & {Z_2} \\
			& A & B && A & {C.}
			\arrow["{\varphi_1}"', from=1-1, to=2-2]
			\arrow["{\Z(\pi_B) \chi}", curve={height=-12pt}, from=1-1, to=2-3]
			\arrow["k"', curve={height=12pt}, from=1-1, to=3-2]
			\arrow["{\varphi_2}"', from=1-4, to=2-5]
			\arrow["{\Z(\pi_C)\chi}", curve={height=-12pt}, from=1-4, to=2-6]
			\arrow["k"', curve={height=12pt}, from=1-4, to=3-5]
			\arrow["{\chi_1}", from=2-2, to=2-3]
			\arrow["{k_1}"', tail, from=2-2, to=3-2]
			\arrow["\lrcorner"{anchor=center, pos=0.125}, draw=none, from=2-2, to=3-3]
			\arrow["{(A)}"{description}, draw=none, from=2-2, to=3-3]
			\arrow["{\varepsilon_B}", tail, from=2-3, to=3-3]
			\arrow["{\chi_2}", from=2-5, to=2-6]
			\arrow["{k_2}"', tail, from=2-5, to=3-5]
			\arrow["\lrcorner"{anchor=center, pos=0.125}, draw=none, from=2-5, to=3-6]
			\arrow["{(B)}"{description}, draw=none, from=2-5, to=3-6]
			\arrow["{\varepsilon_C}", tail, from=2-6, to=3-6]
			\arrow["f"', from=3-2, to=3-3]
			\arrow["g"', from=3-5, to=3-6]
		\end{tikzcd}\]
		Since $k_1 \varphi_1 = k_2 \varphi_2 = k$, there exists a unique $\varphi \colon K \to K_1 \cap K_2$ such that $i_j \varphi = \varphi_j$ for $j \in \{1,2\}$:
		\[\begin{tikzcd}
			K && \\
			& {K_1 \cap K_2} & {K_2} \\
			& {K_1} & A.
			\arrow["\varphi"', from=1-1, to=2-2]
			\arrow["{\varphi_2}", curve={height=-12pt}, from=1-1, to=2-3]
			\arrow["{\varphi_1}"', curve={height=12pt}, from=1-1, to=3-2]
			\arrow["{i_2}", tail, from=2-2, to=2-3]
			\arrow["{i_1}"', tail, from=2-2, to=3-2]
			\arrow["\lrcorner"{anchor=center, pos=0.125}, draw=none, from=2-2, to=3-3]
			\arrow["{k_2}", tail, from=2-3, to=3-3]
			\arrow["{k_1}"', tail, from=3-2, to=3-3]
		\end{tikzcd}\]
		Our goal now is to define the inverse of $\T(\varphi)$. Let $t_{B \times C} \colon \T(B \times C) \to B \times C$ denote the $(B \times C)$-component of the counit of the coreflection of $\tort$ in $\mathcal{C}$. Since $\torf$ is a reflective subcategory of $\mathcal{C}$, hence it is closed under products, and $\zeros \subseteq \torf$, we have $Z_1 \times Z_2 \in \torf$, so that $\T(Z_1 \times Z_2) \in \zeros$. Thanks to this, there exists a unique morphism $\theta$ making the following diagram commute:
		\[\begin{tikzcd}[column sep=1.5cm]
			{\T(Z_1 \times Z_2)} & {\T(B \times C)} & {B \times C} \\
			Z.
			\arrow["{\T(\varepsilon_B \times \varepsilon_C)}", from=1-1, to=1-2]
			\arrow["\theta"', from=1-1, to=2-1]
			\arrow["{t_{B \times C}}", from=1-2, to=1-3]
			\arrow["{\varepsilon_{B \times C}}"', tail, from=2-1, to=1-3]
		\end{tikzcd}\]
		Applying $\T$ to the diagram above (and recalling that, up to isomorphisms, $\T(t_{B \times C}) = 1_{\T(B \times C)}$, $\T(\T(\varepsilon_B \times \varepsilon_C)) = \T(\varepsilon_B \times \varepsilon_C)$, and $\T(\theta) = \theta$), we get $\T(\varepsilon_{B \times C}) \theta = \T(\varepsilon_B \times \varepsilon_C)$. Now, applying $\T$ to the arrow $\langle \chi_1 i_1, \chi_2 i_2 \rangle \colon K_1 \cap K_2 \to Z_1 \times Z_2$ and composing with $\theta$, we obtain $\theta \T(\langle \chi_1 i_1, \chi_2 i_2 \rangle) \colon \T(K_1 \cap K_2) \to Z$. Consider the following diagram in $\tort$, where the inner square is a pullback (since $\T$ preserves limits):
		\[\begin{tikzcd}
			{\T(K)} &&& \\
			& {\T(K_1 \cap K_2)} \\
			&& {\T(K)} & {\T(Z)=Z} \\
			&& {\T(A)} & {\T(B \times C).}
			\arrow["{\T(\varphi)}", from=1-1, to=2-2]
			\arrow[""{name=0, anchor=center, inner sep=0}, "{\T(\chi)}", curve={height=-24pt}, from=1-1, to=3-4]
			\arrow[""{name=1, anchor=center, inner sep=0}, "{\T(k)}"', curve={height=24pt}, from=1-1, to=4-3]
			\arrow["\psi"{description}, from=2-2, to=3-3]
			\arrow["{\theta \T(\spigolose{\chi_1i_1, \chi_2i_2})}"{description}, curve={height=-12pt}, from=2-2, to=3-4]
			\arrow["{\T(k_1i_1)=\T(k_2i_2)}"{description}, curve={height=12pt}, from=2-2, to=4-3]
			\arrow["{\T(\chi)}", from=3-3, to=3-4]
			\arrow["{\T(k)}"', from=3-3, to=4-3]
			\arrow["\lrcorner"{anchor=center, pos=0.125}, draw=none, from=3-3, to=4-4]
			\arrow["{\T(\varepsilon_{B \times C})}", from=3-4, to=4-4]
			\arrow["{\T(\spigolose{f,g})}"', from=4-3, to=4-4]
			\arrow["(\clubsuit)"{description}, draw=none, from=0, to=2-2]
			\arrow["(\spadesuit)"{description}, draw=none, from=1, to=2-2]
		\end{tikzcd}\]
		The existence of $\psi$ is guaranteed since:
		\begin{align*}
			\T(\varepsilon_{B \times C}) \theta \T(\spigolose{\chi_1 i_1, \chi_2 i_2}) &= \T(\varepsilon_B \times \varepsilon_C) \T(\spigolose{\chi_1 i_1, \chi_2 i_2}) \\
			&= \T(\spigolose{\varepsilon_B \chi_1 i_1, \varepsilon_C \chi_2 i_2}) \\
			&= \T(\spigolose{f k_1 i_1, g k_2 i_2}) \\
			&= \T(\spigolose{f,g}) \T(k_1 i_1) \quad \text{(because $k_1 i_1 = k_2 i_2$).}
		\end{align*}
		We now prove that the outer square commutes. We must show that $\T(\spigolose{f,g}) \T(k) = \T(\varepsilon_{B \times C}) \T(\chi)$. Recalling that $\pi_B \varepsilon_{B \times C} = \varepsilon_B \Z(\pi_B)$, we obtain $\T(\pi_B) \T(\varepsilon_{B \times C}) = \T(\varepsilon_B \Z(\pi_B))$. Similarly, one gets $\T(\pi_C) \T(\varepsilon_{B \times C}) = \T(\varepsilon_C \Z(\pi_C))$. Since $\T$ preserves products, it suffices to show that $\T(f k) = \T(\varepsilon_B \Z(\pi_B) \chi)$ and $\T(g k) = \T(\varepsilon_C \Z(\pi_C) \chi)$. However, these last identities hold even before applying $\T$, as observed at the beginning of the proof. Moreover, $(\clubsuit)$ commutes because, for every $T \in \tort$ and for every $Z \in \zeros$, the set $\Hom(T,Z)$ has at most one element. Furthermore, $(\spadesuit)$ commutes since $k_1 i_1 \varphi = k_1 \varphi_1 = k$. Therefore, by the universal property of the inner pullback diagram, $\psi \T(\varphi) = 1_{\T(K)}$. The arrow $k_1 i_1$ is a monomorphism, being a composition of monomorphisms; thus, $\T(k_1 i_1) = \T(k) \psi$ is a monomorphism in $\tort$. Since the composition $\T(k)\psi$ is a monomorphism, its first factor $\psi$ must also be a monomorphism. Having proved that $\psi$ is a split epimorphism and a monomorphism, we conclude that it is an isomorphism.
	\end{proof}
	
	The result we just proved allows us to show that Condition (N) passes to the torsion theory in $\arrc$:
	
	\begin{prop}
		If $\torsione$ satisfies Conditions (N) and (M), then $(\tort_1, \torf_1)$ satisfies Condition (N), too.
	\end{prop}
	
	\begin{proof}
		Consider a morphism $(f,f') \colon a \to b$ in $\arrc$. Computing the $\arrz$-kernel of $(f, f')$ yields the following commutative diagram:
		\[\begin{tikzcd}
			{K[f]} & A & B \\
			{K[f']} & {A'} & {B'}
			\arrow["{k(f)}", from=1-1, to=1-2]
			\arrow["h"', from=1-1, to=2-1]
			\arrow["f", from=1-2, to=1-3]
			\arrow["a", from=1-2, to=2-2]
			\arrow["b", from=1-3, to=2-3]
			\arrow["{k(f')}"', from=2-1, to=2-2]
			\arrow["{f'}"', from=2-2, to=2-3]
		\end{tikzcd}\]
		where $h$ is induced by the universal property of the $\zeros$-kernel. Consider the pullback diagrams:
		\[\begin{tikzcd}
			{K[a] \cap K[f]} & {K[a]} && {K[a]} & {\Z(A')} \\
			{K[f]} & {A,} && A & {A'.}
			\arrow["{i_a}", tail, from=1-1, to=1-2]
			\arrow["{i_f}"', tail, from=1-1, to=2-1]
			\arrow["\lrcorner"{anchor=center, pos=0.125}, draw=none, from=1-1, to=2-2]
			\arrow["{k(a)}", tail, from=1-2, to=2-2]
			\arrow["\chi", from=1-4, to=1-5]
			\arrow["{k(a)}"', tail, from=1-4, to=2-4]
			\arrow["\lrcorner"{anchor=center, pos=0.125}, draw=none, from=1-4, to=2-5]
			\arrow["{k(f')\varepsilon}", tail, from=1-5, to=2-5]
			\arrow["{k(f)}"', tail, from=2-1, to=2-2]
			\arrow["a"', from=2-4, to=2-5]
		\end{tikzcd}\]
		Let $\varepsilon$ denote the inclusion of $\Z(A')$ into $K[f']$, and consider the diagram:
		\[\begin{tikzcd}
			{K[a] \cap K[f]} & {\Z(A')} \\
			{K[f]} & {K[f'].}
			\arrow["{\chi i_a}", from=1-1, to=1-2]
			\arrow["{i_f}"', from=1-1, to=2-1]
			\arrow["{(\heartsuit)}"{description}, draw=none, from=1-1, to=2-2]
			\arrow["\varepsilon", tail, from=1-2, to=2-2]
			\arrow["h"', from=2-1, to=2-2]
		\end{tikzcd} \]
		We prove that the square $(\heartsuit)$ is a pullback. First, we show that it commutes. Observe that $k(f') h i_f = a k(f) i_f = a k(a) i_a = k(f') \varepsilon \chi i_a$. As $k(f')$ is a monomorphism, we obtain $h i_f = \varepsilon \chi i_a$. Now, consider an object $X$ and a pair of morphisms $x_1 \colon X \to K[f]$ and $x_2 \colon X \to \Z(A')$ such that $h x_1 = \varepsilon x_2$. From the equalities $a k(f) x_1 = k(f') h x_1 = k(f') \varepsilon x_2$, it follows that there exists a unique $\varphi \colon X \to K[a]$ such that $k(a) \varphi = k(f) x_1$ and $\chi \varphi = x_2$. Furthermore, the relation $k(a) \varphi = k(f) x_1$ guarantees the existence of a unique $\theta \colon X \to K[a] \cap K[f]$ such that $i_f \theta = x_1$ and $i_a \theta = \varphi$. Finally, we have that $i_f \theta = x_1$ and $\chi i_a \theta = \chi \varphi = x_2$, and the fact that $i_f$ is a monomorphism gives us that the morphism $\theta$ satisfying these two equalities is unique, allowing us to conclude that $(\heartsuit)$ is indeed a pullback. This shows that $K[h] \cong K[a] \cap K[f]$.
		
		Recalling that $\Ht(h) \colon \T(K[h]) \to \Z(A')$, to show that $(\tort_1, \torf_1)$ satisfies (N) it suffices to prove that $\T(K[h])$ is the $\zeros$-kernel of an arrow with domain $A$. Indeed, if it is the case, then $\T(K[h])$ is the $\zeros$-kernel of its $\zeros$-cokernel $e \colon A \to Q$ in $\mathcal{C}$ (whose existence is guaranteed by Corollary $2.15$ in \cite{CM torsion theories 1}). Since $\arrz$-kernels are computed levelwise, $\T_1(h)$ is then the $\arrz$-kernel of $(e, 1_{A'})$ in $\arrc$:
		\[\begin{tikzcd}
			{\T(K[h])} & A & Q \\
			{\Z(A')} & {A'} & {A'.}
			\arrow[tail, from=1-1, to=1-2]
			\arrow["\T_1(h)"', from=1-1, to=2-1]
			\arrow["e", from=1-2, to=1-3]
			\arrow["a", from=1-2, to=2-2]
			\arrow["!", from=1-3, to=2-3]
			\arrow["{\varepsilon_{A'}}"', tail, from=2-1, to=2-2]
			\arrow["{1_{A'}}"', from=2-2, to=2-3]
		\end{tikzcd}\]
		Observe that $\T(K[h]) \cong \T(K[a] \cap K[f]) \cong \T(K[\spigolose{a,f}])$, thanks to Proposition \ref{prop torsion of inters of ker}. Therefore, the validity of (N) for $\torsione$ implies the validity of (N) for $(\tort_1, \torf_1)$.
	\end{proof}
	
	\section{Condition (P) for the torsion theory in the category of arrows}
	We now address the question of the validity of Condition (P) for the torsion theory $(\tort_1, \torf_1)$. In order to do that, we need to rephrase the property for a functor to be protoadditive in the sense of Definition \ref{def protoadditive}.
	
	Let $\G \colon \mathcal{C} \to \mathcal{D}$ be a functor between regular protomodular categories, both equipped with a class of zero objects. With a slight abuse of notation, we use the same symbol $\Z$ to denote the associated zero-object coreflector in both categories, and we write $\G\Z=\Z\G$ meaning that, for every object $C$ in $\mathcal{C}$, the image $\G(\Z(C))$ under $\G$ of the zero part of $C$ coincides with the zero part $\Z(\G(C))$ of $\G(C)$ in $\mathcal{D}$, and similarly for morphisms $f \colon C \to C'$ in $\mathcal{C}$.
	
	\begin{prop}
		Given a functor $\G \colon \mathcal{C} \to \mathcal{D}$ between regular protomodular categories with a class of zero objects, suppose that $\G\Z=\Z\G$. Then $\G$ is protoadditive if and only if $\G$ preserves $\zeros$-kernels of split epimorphisms.
	\end{prop}
	
	\begin{proof}
		Suppose first that $\G$ is protoadditive. Let $f \colon A \to B$ be a split epimorphism with section $s$, and let $k = \Zker(f)$, as pictured in the pullback diagram
		\[\begin{tikzcd}
			K & {\Z(B)} \\
			A & B.
			\arrow[from=1-1, to=1-2]
			\arrow["k"', from=1-1, to=2-1]
			\arrow["\lrcorner"{anchor=center, pos=0.125}, draw=none, from=1-1, to=2-2]
			\arrow["{\varepsilon_B}", from=1-2, to=2-2]
			\arrow["f"', shift right, from=2-1, to=2-2]
			\arrow["s"', shift right, from=2-2, to=2-1]
		\end{tikzcd}\]
		Since $\G$ is protoadditive, it preserves such a pullback. Recalling that $\G(\Z(B)) = \Z(\G(B))$, we get that $\G(k)$ is the $\zeros$-kernel of $\G(f)$.
		
		Conversely, suppose that $\G$ preserves $\zeros$-kernels of split epimorphisms. Let $f \colon A \to B$ be a split epimorphism with section $s$, and let $g \colon C \to B$ be a morphism inverted by $\Z$. Let $P$ be the pullback of $f$ along $g$, and let $k_1 = \Zker(l)$ and $k_2 = \Zker(f)$:
		\[\begin{tikzcd}
			{K[l]} & P & C \\
			{K[f]} & A & B.
			\arrow["{k_1}", from=1-1, to=1-2]
			\arrow["{k'}"', from=1-1, to=2-1]
			\arrow["l"', shift right, from=1-2, to=1-3]
			\arrow["h"', from=1-2, to=2-2]
			\arrow["\lrcorner"{anchor=center, pos=0.125}, draw=none, from=1-2, to=2-3]
			\arrow["r"', shift right, from=1-3, to=1-2]
			\arrow["g", from=1-3, to=2-3]
			\arrow["{k_2}"', from=2-1, to=2-2]
			\arrow["f"', shift right, from=2-2, to=2-3]
			\arrow["s"', shift right, from=2-3, to=2-2]
		\end{tikzcd}\]
		Since $\Z(g)$ is an isomorphism, we conclude that the induced morphism $k'$ is an isomorphism using Proposition $4.3$ in \cite{Cappelletti}. Moreover, since by assumption $\G$ preserves $\zeros$-kernels of split epimorphisms, $\G(k') \colon K[\G(l)] \to K[\G(f)]$ is an isomorphism. Consider the pullback of $\G(f)$ along $\G(g)$ and the corresponding commutative diagram:
		\[\begin{tikzcd}
			{K[\G(l)]} & {\G(P)} & {\G(C)} \\
			{K[\pi_C]} & Q & {\G(C)} \\
			{K[\G(f)]} & {\G(A)} & {\G(B),}
			\arrow["{\G(k_1)}", from=1-1, to=1-2]
			\arrow["u"', from=1-1, to=2-1]
			\arrow["{\G(k')}"', curve={height=30pt}, from=1-1, to=3-1]
			\arrow["{\G(l)}"', shift right, from=1-2, to=1-3]
			\arrow["\varphi"', from=1-2, to=2-2]
			\arrow["{\G(r)}"', shift right, from=1-3, to=1-2]
			\arrow[equal, from=1-3, to=2-3]
			\arrow["{\overline{k}}", from=2-1, to=2-2]
			\arrow["v"', from=2-1, to=3-1]
			\arrow["{\pi_C}", from=2-2, to=2-3]
			\arrow["{\pi_A}"', from=2-2, to=3-2]
			\arrow["\lrcorner"{anchor=center, pos=0.125}, draw=none, from=2-2, to=3-3]
			\arrow["{\G(g)}", from=2-3, to=3-3]
			\arrow["{\G(k_2)}"', from=3-1, to=3-2]
			\arrow["{\G(f)}"', shift right, from=3-2, to=3-3]
			\arrow["{\G(s)}"', shift right, from=3-3, to=3-2]
		\end{tikzcd}\]
		where $\overline{k}$ is the $\zeros$-kernel of $\pi_C$, while $u, v, \varphi$ are induced by the universal properties of the $\zeros$-kernels and of the pullback. Since $\Z(\G(g)) = \G(\Z(g))$, the morphism $\G(g)$ is inverted by $\Z$. By Proposition $4.3$ in \cite{Cappelletti}, it follows that $v$ is an isomorphism. Consequently, since $\G(k') = v u$, $u$ is an isomorphism, too. Finally, we can apply Proposition $5.2$ in \cite{Cappelletti} to deduce that $\varphi$ is an isomorphism, and therefore $\G$ is protoadditive.
	\end{proof}
	
	Let now $\torsione$ be a torsion theory in a regular protomodular category $\mathcal{C}$ with a class of zero objects. With another abuse of notation, let us denote by $\T$ not only the coreflector $\T \colon \mathcal{C} \to \tort$ but also its composite $\T \colon \mathcal{C} \to \mathcal{C}$ with the inclusion functor, which is usually called \emph{radical}. It is easy to observe that $\T \Z = \Z \T$, because for all $A \in \mathcal{C}$ the arrow $t_A \colon \T(A) \to A$ is a monomorphism and $\Z$ inverts monomorphisms, so that $\Z \T(A) = \Z(A) = \T \Z(A)$. The previous proposition, then, immediately implies the following:
	
	\begin{cor}\label{T protoadditivo}
		The radical $\T \colon \mathcal{C} \to \mathcal{C}$ is protoadditive if and only if it preserves $\zeros$-kernels of split epimorphisms.
	\end{cor}
	
	We recall that a subcategory $\tort$ of a category $\mathcal{C}$ is said to be \emph{hereditary} with respect to a class $\mathbb{M}$ of monomorphisms in $\mathcal{C}$ if, for every $f \colon A \to B$ in $\mathbb{M}$, $B \in \tort$ implies $A \in \tort$, too. This property, for a suitable $\mathbb{M}$, is related to protoadditivity:
	
	\begin{prop}
		Let $\torsione$ be a torsion theory in a regular protomodular category $\mathcal{C}$ with a class of zero objects. $\tort$ is hereditary with respect to the class of protosplit monomorphisms (i.e.\ $\zeros$-kernels of split epimorphisms) if and only if the radical $\T \colon \mathcal{C} \to \mathcal{C}$ is protoadditive.
	\end{prop}
	
	\begin{proof}
		Suppose $\tort$ is hereditary with respect to the class of protosplit monomorphisms. Thanks to Corollary \ref{T protoadditivo}, it suffices to prove that $\T$ preserves $\zeros$-kernels of split epimorphisms. Let $f \colon A \to B$ be a split epimorphism with section $s$, and let $k = \Zker(f)$, $k' = \Zker(\T(f))$. Since $\T(f) \T(k) = \T(f k) \in \zideal$, there exists a unique morphism $\varphi$ making the following diagram commute:
		\[\begin{tikzcd}
			{\T(K)} & {\T(A)} & {\T(B)} \\
			{K'} & {\Z(\T(B)).}
			\arrow["{\T(k)}", from=1-1, to=1-2]
			\arrow["\varphi"', from=1-1, to=2-1]
			\arrow["{\T(f)}"', shift right, from=1-2, to=1-3]
			\arrow["{\T(s)}"', shift right, from=1-3, to=1-2]
			\arrow["{k'}"', from=2-1, to=1-2]
			\arrow["{\chi'}"', from=2-1, to=2-2]
			\arrow["{\varepsilon_{\T(B)}}"', from=2-2, to=1-3]
		\end{tikzcd}\]
		By assumption, $K' \in \tort$. Now, consider the pullback diagram
		\[\begin{tikzcd}
			K & {\Z(B)} \\
			A & B.
			\arrow["\chi", from=1-1, to=1-2]
			\arrow["k"', from=1-1, to=2-1]
			\arrow["\lrcorner"{anchor=center, pos=0.125}, draw=none, from=1-1, to=2-2]
			\arrow["{\varepsilon_B}", from=1-2, to=2-2]
			\arrow["f"', shift right, from=2-1, to=2-2]
			\arrow["s"', shift right, from=2-2, to=2-1]
		\end{tikzcd}\]
		Since the functor $\T \colon \mathcal{C} \to \tort$ preserves limits, the image of the diagram above under $\T$ is a pullback in $\tort$. Hence, since $K' \in \tort$, there exists a unique morphism $\psi$ making the following diagram commute in $\tort$:
		\[\begin{tikzcd}
			{K'} && \\
			& {\T(K)} & {\T(\Z(B))} \\
			& {\T(A)} & {\T(B).}
			\arrow["\psi", from=1-1, to=2-2]
			\arrow["{\chi'}", curve={height=-12pt}, from=1-1, to=2-3]
			\arrow["{k'}"', curve={height=12pt}, from=1-1, to=3-2]
			\arrow["{\T(\chi)}", from=2-2, to=2-3]
			\arrow["{\T(k)}"', from=2-2, to=3-2]
			\arrow["{\T(\varepsilon_B)}", from=2-3, to=3-3]
			\arrow["{\T(f)}"', shift right, from=3-2, to=3-3]
			\arrow["{\T(s)}"', shift right, from=3-3, to=3-2]
		\end{tikzcd}\]
		It is straightforward to verify that $\varphi$ and $\psi$ are mutually inverse.
		
		Conversely, consider a split epimorphism $g \colon T \to B$, with section $r$, such that $T \in \tort$. Let $\overline{k} \colon \overline{K} \to T$ be the $\zeros$-kernel of $g$. Applying $\F$ to $r$, we obtain a split monomorphism $\F(r) \colon \F(B) \to \F(T)$ with $\F(T) \in \zeros$. Therefore, $\F(B) \in \zeros$ as well, which implies $B \in \tort$. Hence, since $\T$ is protoadditive by assumption, applying $\T$ to the sequence $\overline{K} \xrightarrow{\overline{k}} T \xrightarrow{g} B$, and noting that $\T(g) = g$, we obtain $\T(\overline{K}) \cong \Zker(\T(g)) \cong \Zker(g) \cong \overline{K}$. Consequently, $\overline{K} \in \tort$.
	\end{proof}
	
	\begin{lem}\label{q regolare}
		Let $\torsione$ be a torsion theory in a regular protomodular category $\mathcal{C}$ with a class of zero objects. Consider a split epimorphism $p$ with section $s$ and a morphism $g$. Let $P$ denote the pullback of $p$ along $g$ and $F_P$ the pullback of $\F(p)$ along $\F(g)$. The induced morphism $q \colon P \to F_P$ is a regular epimorphism.
	\end{lem}
	
	\begin{proof}
		Let $P$ be the pullback of $p$ along $g$:
		\[
		\begin{tikzcd}
			P \arrow[r, "p'"] \arrow[d, "g'"'] \arrow[dr, phantom, "\lrcorner", very near start] & C \arrow[d, "g"] \\
			A \arrow[r, "p"'] & B
		\end{tikzcd}
		\]
		and $F_P$ be the pullback of $\F(p)$ along $\F(g)$:
		\[\begin{tikzcd}
			P && \\
			& {F_P} & {\F(C)} \\
			& {\F(A)} & {\F(B).}
			\arrow["q", from=1-1, to=2-2]
			\arrow["{\eta_Cp'}", curve={height=-12pt}, from=1-1, to=2-3]
			\arrow["{\eta_A g'}"', curve={height=12pt}, from=1-1, to=3-2]
			\arrow["{\pi_C}", from=2-2, to=2-3]
			\arrow["{\pi_A}"', from=2-2, to=3-2]
			\arrow["\lrcorner"{anchor=center, pos=0.125}, draw=none, from=2-2, to=3-3]
			\arrow["{\F(g)}", from=2-3, to=3-3]
			\arrow["{\F(p)}"', from=3-2, to=3-3]
		\end{tikzcd}\]
		The morphism $q \colon P \to F_P$, induced by the universal property of $F_P$, is the unique arrow such that $\pi_A q = \eta_A g'$ and $\pi_C q = \eta_C p'$. To prove that $q$ is a regular epimorphism, we construct the pullback diagram
		\begin{equation*}
			\begin{tikzcd}
				I \arrow[r, "\gamma"] \arrow[d, "\alpha"'] \arrow[dr, phantom, "\lrcorner", very near start] & C \arrow[d, "\eta_B g"] \\
				\F(A) \arrow[r, "\F(p)"'] & \F(B).
			\end{tikzcd}
		\end{equation*}
		We observe that $\F(p) \eta_A g' = \eta_B p g' = \eta_B g p'$, so there exists a unique $q_1 \colon P \to I$ such that $\alpha q_1 = \eta_A g'$ and $\gamma q_1 = p'$. Moreover, since $\F(p) \alpha = \eta_B g \gamma = \F(g) \eta_C \gamma$, there exists a unique $q_2 \colon I \to F_P$ such that $\pi_A q_2 = \alpha$ and $\pi_C q_2 = \eta_C \gamma$. The composition $q_2 q_1$ satisfies $\pi_A q_2 q_1 = \alpha q_1 = \eta_A g' = \pi_A q$ and $\pi_C q_2 q_1 = \eta_C \gamma q_1 = \eta_C p' = \pi_C q$; hence, $q = q_2 q_1$.
		
		First, we establish that $q_1$ is a regular epimorphism. Consider the pullback diagram
		\[\begin{tikzcd}
			A && \\
			& {P^*} & {\F(A)} \\
			& B & {\F(B)}
			\arrow["\chi", from=1-1, to=2-2]
			\arrow["{\eta_A}", curve={height=-12pt}, from=1-1, to=2-3]
			\arrow["p"', curve={height=12pt}, from=1-1, to=3-2]
			\arrow["{p_1}", from=2-2, to=2-3]
			\arrow["{p_2}"', from=2-2, to=3-2]
			\arrow["\lrcorner"{anchor=center, pos=0.125}, draw=none, from=2-2, to=3-3]
			\arrow["{\F(p)}", from=2-3, to=3-3]
			\arrow["{\eta_B}"', from=3-2, to=3-3]
		\end{tikzcd}\]
		and the morphism $\chi$ induced by the universal property. Since $\mathcal{C}$ is regular protomodular (and thus Mal'cev, see e.g. Proposition $3.1.19$ in \cite{Bbbook}), we can use Proposition $3.2$ of \cite{denormalized 3x3} to conclude that $\chi$ is a regular epimorphism. Since $\eta_B g \gamma = \F(p) \alpha$, there exists a unique $h \colon I \to P^*$ such that $p_1 h = \alpha$ and $p_2 h = g \gamma$. We claim the following square is a pullback:
		\begin{equation*}
			\begin{tikzcd}
				P \arrow[r, "q_1"] \arrow[d, "g'"'] & I \arrow[d, "h"] \\
				A \arrow[r, "\chi"'] & P^*.
			\end{tikzcd}
		\end{equation*}
		This square commutes; indeed, $p_1 h q_1 = \alpha q_1 = \eta_A g'= p_1 \chi g'$ and $p_2 h q_1 = g \gamma q_1 = g p' = p g' = p_2 \chi g'$ (being $p_1, p_2$ jointly monomorphic). Consider the commutative diagram
		\[\begin{tikzcd}
			P & I & C \\
			A & {P^*} & B
			\arrow["{q_1}", from=1-1, to=1-2]
			\arrow["{p'}", curve={height=-30pt}, from=1-1, to=1-3]
			\arrow["{g'}"', from=1-1, to=2-1]
			\arrow["\gamma", from=1-2, to=1-3]
			\arrow["h", from=1-2, to=2-2]
			\arrow["{(\diamond)}"{description}, draw=none, from=1-2, to=2-3]
			\arrow["g", from=1-3, to=2-3]
			\arrow["\chi"', from=2-1, to=2-2]
			\arrow["p"', curve={height=30pt}, from=2-1, to=2-3]
			\arrow["{p_2}"', from=2-2, to=2-3]
		\end{tikzcd}\]
		and observe that the whole rectangle is a pullback diagram. If we can prove that $(\diamond)$ is a pullback diagram, we can deduce that the square on the left is also a pullback diagram. Consequently, $q_1$ would be a regular epimorphism, being the pullback of the regular epimorphism $\chi$ in the regular category $\mathcal{C}$. Therefore, consider the commutative diagram
		\[\begin{tikzcd}
			I & C \\
			{P^*} & B \\
			{\F(A)} & {\F(B)}
			\arrow["\gamma", from=1-1, to=1-2]
			\arrow["h"', from=1-1, to=2-1]
			\arrow["{(\diamond)}"{description}, draw=none, from=1-1, to=2-2]
			\arrow["\alpha"', curve={height=30pt}, from=1-1, to=3-1]
			\arrow["g", from=1-2, to=2-2]
			\arrow["{\eta_Bg}", curve={height=-30pt}, from=1-2, to=3-2]
			\arrow["{p_2}", from=2-1, to=2-2]
			\arrow["{p_1}"', from=2-1, to=3-1]
			\arrow["{\eta_B}", from=2-2, to=3-2]
			\arrow["{\F(p)}"', from=3-1, to=3-2]
		\end{tikzcd}\]
		and note that both the entire rectangle and the bottom square are pullback diagrams. From that, we conclude that $(\diamond)$ is indeed a pullback diagram.
		
		Next, consider the commutative diagram
		\[\begin{tikzcd}
			I & C \\
			{F_P} & {\F(C)} \\
			{\F(A)} & {\F(B),}
			\arrow["\gamma", from=1-1, to=1-2]
			\arrow["{q_2}"', from=1-1, to=2-1]
			\arrow["\alpha"', curve={height=30pt}, from=1-1, to=3-1]
			\arrow["{\eta_C}", from=1-2, to=2-2]
			\arrow["{\eta_B g}", curve={height=-30pt}, from=1-2, to=3-2]
			\arrow["{\pi_C}"', from=2-1, to=2-2]
			\arrow["{\pi_A}"', from=2-1, to=3-1]
			\arrow["{\F(g)}", from=2-2, to=3-2]
			\arrow["{\F(p)}"', from=3-1, to=3-2]
		\end{tikzcd}\]
		where the bottom square and the entire rectangle are pullback diagrams. It follows that the top square is a pullback, too. Hence, $q_2$ is a regular epimorphism as the pullback of the regular epimorphism $\eta_C$.
		
		Finally, since $q = q_2 q_1$ is the composition of two regular epimorphisms, it is a regular epimorphism, as well.
	\end{proof}
	
	\begin{prop}
		Let $\torsione$ be a torsion theory satisfying Condition (M') in a regular protomodular category $\mathcal{C}$ with a class of zero objects. Then, the reflector $\F \colon \mathcal{C} \to \torf$ is protoadditive if and only if the radical $\T \colon \mathcal{C} \to \mathcal{C}$ is protoadditive.
	\end{prop}
	
	\begin{proof}
		Let $p$ be a split epimorphism with section $s$. Consider the pullback of $p$ along an arrow $g$ in $\mathcal{C}$ inverted by $\Z$:
		\[\begin{tikzcd}
			P & C \\
			A & B.
			\arrow["{p'}"', shift right, from=1-1, to=1-2]
			\arrow["{g'}"', from=1-1, to=2-1]
			\arrow["\lrcorner"{anchor=center, pos=0.125}, draw=none, from=1-1, to=2-2]
			\arrow["{s'}"', shift right, from=1-2, to=1-1]
			\arrow["g", from=1-2, to=2-2]
			\arrow["p"', shift right, from=2-1, to=2-2]
			\arrow["s"', shift right, from=2-2, to=2-1]
		\end{tikzcd}\]
		Consider then the pullbacks of $\F(p)$ along $\F(g)$ and of $\T(p)$ along $\T(g)$ in $\mathcal{C}$:
		\[\begin{tikzcd}
			{F_P} & \F(C) & {T_P} & \T(C) \\
			\F(A) & \F(B) & \T(A) & \T(B).
			\arrow["{\pi_C}"', shift right, from=1-1, to=1-2]
			\arrow["{\pi_A}"', from=1-1, to=2-1]
			\arrow["\lrcorner"{anchor=center, pos=0.125}, draw=none, from=1-1, to=2-2]
			\arrow["\sigma"', shift right, from=1-2, to=1-1]
			\arrow["\F(g)", from=1-2, to=2-2]
			\arrow["{\rho_C}", from=1-3, to=1-4]
			\arrow["{\rho_A}"', from=1-3, to=2-3]
			\arrow["\lrcorner"{anchor=center, pos=0.125}, draw=none, from=1-3, to=2-4]
			\arrow["\T(g)", from=1-4, to=2-4]
			\arrow["\F(p)"', shift right, from=2-1, to=2-2]
			\arrow["\F(s)"', shift right, from=2-2, to=2-1]
			\arrow["\T(p)"', shift right, from=2-3, to=2-4]
			\arrow["\T(s)"', shift right, from=2-4, to=2-3]
		\end{tikzcd}\]
		Since $p t_A \rho_A = t_B \T(p) \rho_A = t_B \T(g) \rho_C = g t_C \rho_C$, there exists a unique $k \colon T_P \to P$ such that $g' k = t_A \rho_A$ and $p' k = t_C \rho_C$. Furthermore, since $\F(p) \eta_A g' = \eta_B p g' = \eta_B g p' = \F(g) \eta_C p'$, there exists a unique $q \colon P \to F_P$ such that $\pi_A q = \eta_A g'$ and $\pi_C q = \eta_C p'$:
		\[\begin{tikzcd}
			P &&& {T_P} && \\
			& {F_P} & \F(C) && P & C \\
			& \F(A) & \F(B), && A & B.
			\arrow["q", from=1-1, to=2-2]
			\arrow["{\eta_C p'}", curve={height=-12pt}, from=1-1, to=2-3]
			\arrow["{\eta_A g'}"', curve={height=12pt}, from=1-1, to=3-2]
			\arrow["k", from=1-4, to=2-5]
			\arrow["{t_C \rho_C}", curve={height=-12pt}, from=1-4, to=2-6]
			\arrow["{t_A \rho_A}"', curve={height=12pt}, from=1-4, to=3-5]
			\arrow["{\pi_C}"', shift right, from=2-2, to=2-3]
			\arrow["{\pi_A}"', from=2-2, to=3-2]
			\arrow["\lrcorner"{anchor=center, pos=0.125}, draw=none, from=2-2, to=3-3]
			\arrow["\sigma"', shift right, from=2-3, to=2-2]
			\arrow["{\F(g)}", from=2-3, to=3-3]
			\arrow["{\F(p)}"', shift right, from=3-2, to=3-3]
			\arrow["{\F(s)}"', shift right, from=3-3, to=3-2]
			\arrow["{p'}"', shift right, from=2-5, to=2-6]
			\arrow["{s'}"', shift right, from=2-6, to=2-5]
			\arrow["{g'}"', from=2-5, to=3-5]
			\arrow["\lrcorner"{anchor=center, pos=0.125}, draw=none, from=2-5, to=3-6]
			\arrow["g", from=2-6, to=3-6]
			\arrow["p"', shift right, from=3-5, to=3-6]
			\arrow["s"', shift right, from=3-6, to=3-5]
		\end{tikzcd}\]
		For any object $X \in \mathcal{C}$, we have that $\eta_{\T(X)} \colon \T(X) \to \F(\T(X)) = \Z(\F(X))$ thanks to Condition (M'). Clearly, $\varepsilon_{\F(X)} \eta_{\T(X)} = \eta_X t_X$. Given an arrow $f \colon X \to Y$, consider the diagram
		\[\begin{tikzcd}
			\T(X) & X & \F(X) \\
			& \Z(\F(X)) \\
			& \Z(\F(Y)) \\
			\T(Y) & Y & \F(Y)
			\arrow["{t_X}", from=1-1, to=1-2]
			\arrow["{\eta_{\T(X)}}"', from=1-1, to=2-2]
			\arrow[""{name=0, anchor=center, inner sep=0}, "\T(f)"', from=1-1, to=4-1]
			\arrow["{\eta_X}", from=1-2, to=1-3]
			\arrow[""{name=1, anchor=center, inner sep=0}, "\F(f)", from=1-3, to=4-3]
			\arrow["{\varepsilon_{\F(X)}}"', from=2-2, to=1-3]
			\arrow[""{name=2, anchor=center, inner sep=0}, "\Z(\F(f))"{description}, from=2-2, to=3-2]
			\arrow["{\varepsilon_{\F(Y)}}", from=3-2, to=4-3]
			\arrow["{\eta_{\T(Y)}}", from=4-1, to=3-2]
			\arrow["{t_Y}"', from=4-1, to=4-2]
			\arrow["{\eta_Y}"', from=4-2, to=4-3]
			\arrow["{(1)}"{marking, allow upside down}, draw=none, from=0, to=2]
			\arrow["{(2)}"{marking, allow upside down}, draw=none, from=2, to=1]
		\end{tikzcd}\]
		and note that the square (2) commutes by naturality of $\varepsilon$, while square (1) commutes by naturality of $\eta$. We now prove that the following square is a pullback in $\mathcal{C}$:
		\[\begin{tikzcd}[column sep=2cm]
			{T_P} & {\Z(F_P)} \\
			P & {F_P.}
			\arrow["{\Z(\sigma) \eta_{\T(C)} \rho_C}", from=1-1, to=1-2]
			\arrow["k"', from=1-1, to=2-1]
			\arrow["{\varepsilon_{F_P}}", from=1-2, to=2-2]
			\arrow["q"', from=2-1, to=2-2]
		\end{tikzcd}\]
		We start by showing that the square commutes. By definition, $\pi_A q k = \eta_A g' k = \eta_A t_A \rho_A = \varepsilon_{\F(A)} \eta_{\T(A)} \rho_A$. From $\T(p) \rho_A = \T(g) \rho_C$, we deduce:
		\begin{align*}
			&\T(p) \rho_A = \T(g) \rho_C \\
			&\eta_{\T(B)} \T(p) \rho_A = \eta_{\T(B)} \T(g) \rho_C && \text{(applying $\eta_{\T(B)}$ on the left)} \\
			&\Z(\F(p)) \eta_{\T(A)} \rho_A = \Z(\F(g)) \eta_{\T(C)} \rho_C && \text{(by commutativity of (1))} \\
			&\eta_{\T(A)} \rho_A = \inv{\Z(\F(p))} \Z(\F(g)) \eta_{\T(C)} \rho_C && \text{(since $\Z(\F(p))$ is an isomorphism).}
		\end{align*}
		Substituting back, we obtain $\pi_A q k = \varepsilon_{\F(A)} \inv{\Z(\F(p))} \Z(\F(g)) \eta_{\T(C)} \rho_C$. On the other hand,
		\[ \pi_A \varepsilon_{F_P} \Z(\sigma) \eta_{\T(C)} \rho_C = \varepsilon_{\F(A)} \Z(\pi_A) \Z(\sigma) \eta_{\T(C)} \rho_C = \varepsilon_{\F(A)} \inv{\Z(\F(p))} \Z(\F(g)) \eta_{\T(C)} \rho_C, \]
		where the last equality holds because $\Z(\pi_A) \Z(\sigma) = \inv{\Z(\F(p))} \Z(\F(g))$ as they are arrows between zero objects with the same domain and codomain. Now, $\pi_C q k = \eta_C p' k = \eta_C t_C \rho_C$, while $\pi_C \varepsilon_{F_P} \Z(\sigma) \eta_{\T(C)} \rho_C = \varepsilon_{\F(C)} \Z(\pi_C) \Z(\sigma) \eta_{\T(C)} \rho_C = \varepsilon_{\F(C)} \eta_{\T(C)} \rho_C$ (since $\Z(\pi_C)$ and $\Z(\sigma)$ are mutually inverse). Since $\pi_A$ and $\pi_C$ are jointly monomorphic, the square commutes.
		
		To prove that it is a pullback, consider an object $X$ and the diagram
		\[\begin{tikzcd}[column sep=2cm]
			X && \\
			& {T_P} & {\Z(F_P)} \\
			& P & {F_P}
			\arrow["\varphi", dashed, from=1-1, to=2-2]
			\arrow["y", curve={height=-12pt}, from=1-1, to=2-3]
			\arrow["x"', curve={height=12pt}, from=1-1, to=3-2]
			\arrow["{\Z(\sigma) \eta_{\T(C)} \rho_C}", from=2-2, to=2-3]
			\arrow["k"', from=2-2, to=3-2]
			\arrow["{\varepsilon_{F_P}}", from=2-3, to=3-3]
			\arrow["q"', from=3-2, to=3-3]
		\end{tikzcd}\]
		where $q x = \varepsilon_{F_P} y$. We want to show the existence and uniqueness of the dashed arrow $\varphi$ making the triangles commute. Since $\pi_A q x = \pi_A \varepsilon_{F_P} y$, it follows that $\eta_A g' x = \varepsilon_{\F(A)} \Z(\pi_A) y$. Therefore, there exists a unique $\varphi_A \colon X \to \T(A)$ such that $t_A \varphi_A = g' x$ and $\eta_{\T(A)} \varphi_A = \Z(\pi_A) y$:
		\[\begin{tikzcd}
			X && \\
			& \T(A) & \Z(\F(A)) \\
			& A & \F(A).
			\arrow["{\varphi_A}", from=1-1, to=2-2]
			\arrow["{\Z(\pi_A) y}", curve={height=-12pt}, from=1-1, to=2-3]
			\arrow["{g' x}"', curve={height=12pt}, from=1-1, to=3-2]
			\arrow["{\eta_{\T(A)}}", from=2-2, to=2-3]
			\arrow["{t_A}"', from=2-2, to=3-2]
			\arrow["\lrcorner"{anchor=center, pos=0.125}, draw=none, from=2-2, to=3-3]
			\arrow["{\varepsilon_{\F(A)}}", from=2-3, to=3-3]
			\arrow["{\eta_A}"', from=3-2, to=3-3]
		\end{tikzcd}\]
		Similarly, there exists a unique $\varphi_C \colon X \to \T(C)$ such that $t_C \varphi_C = p' x$ and $\eta_{\T(C)} \varphi_C = \Z(\pi_C) y$. Note that $t_B \T(p) \varphi_A = p t_A \varphi_A = p g' x$ and $t_B \T(g) \varphi_C = g t_C \varphi_C = g p' x = p g' x$. Since $t_B$ is a monomorphism, we obtain $\T(p) \varphi_A = \T(g) \varphi_C$. Thus, there exists a unique $\varphi \colon X \to T_P$ such that $\rho_A \varphi = \varphi_A$ and $\rho_C \varphi = \varphi_C$:
		\[\begin{tikzcd}
			X && \\
			& {T_P} & \T(C) \\
			& \T(A) & \T(B).
			\arrow["\varphi", from=1-1, to=2-2]
			\arrow["{\varphi_C}", curve={height=-12pt}, from=1-1, to=2-3]
			\arrow["{\varphi_A}"', curve={height=12pt}, from=1-1, to=3-2]
			\arrow["{\rho_C}", from=2-2, to=2-3]
			\arrow["{\rho_A}"', from=2-2, to=3-2]
			\arrow["\lrcorner"{anchor=center, pos=0.125}, draw=none, from=2-2, to=3-3]
			\arrow["{\T(g)}", from=2-3, to=3-3]
			\arrow["{\T(p)}"', from=3-2, to=3-3]
		\end{tikzcd}\]
		Observe that $g' k \varphi = t_A \rho_A \varphi = t_A \varphi_A = g' x$ and $p' k \varphi = t_C \rho_C \varphi = t_C \varphi_C = p' x$; recalling that $g'$ and $p'$ are jointly monomorphic, we deduce that $k \varphi = x$. Moreover, $\Z(\sigma) \eta_{\T(C)} \rho_C \varphi = \Z(\sigma) \eta_{\T(C)} \varphi_C = \Z(\sigma) \Z(\pi_C) y = y$. Finally, to show that $\varphi$ is unique, consider any $\psi \colon X \to T_P$ such that $k \psi = x$ and $\Z(\sigma) \eta_{\T(C)} \rho_C \psi = y$. We have $t_A \rho_A \psi = g' k \psi = g' x = t_A \varphi_A$. Since $t_A$ is a monomorphism, $\rho_A \psi = \varphi_A$. Similarly, $\rho_C \psi = \varphi_C$. Therefore, $\varphi = \psi$. We have thus proved that $k \colon T_P \to P$ is the $\zeros$-kernel of $q$.
		
		Now, if the reflector $\F \colon \mathcal{C} \to \torf$ is protoadditive, it is straightforward to verify that $q$ is isomorphic to $\eta_P$. Since $k = \Zker(q)$, it follows that $k$ is isomorphic to $t_P$ (and so $\T(P) \cong T_P$). Since $\T \colon \mathcal{C} \to \tort$ preserves limits, we conclude that the radical $\T \colon \mathcal{C} \to \mathcal{C}$ is protoadditive.
		
		Conversely, if $\T \colon \mathcal{C} \to \mathcal{C}$ is protoadditive, it is easy to check that $k$ is isomorphic to $t_P$. Observe that $F_P \in \torf$, as $\torf$ is closed under limits. Since the sequence formed by $k$ and $\eta_P$ is $\zeros$-exact, there exists a unique $\theta \colon \F(P) \to F_P$ such that $\theta \eta_P = q$. Moreover, the morphisms $\Z(\F(p'))$ and $\Z(\F(s'))$ induce an isomorphism $\Z(\F(P)) \cong \Z(\F(C))$. Similarly, the morphisms $\Z(\pi_C)$ and $\Z(\sigma)$ witness that $\Z(F_P) \cong \Z(\F(C))$. Consequently, $\Z(\F(P)) \cong \Z(F_P)$. Furthermore, by Lemma \ref{q regolare}, $q$ is a regular epimorphism; since $\mathcal{C}$ is regular, $\theta$ is a regular epimorphism, too. By applying Lemma $5.4$ of \cite{Cappelletti} to the diagram
		\[\begin{tikzcd}
			{T_P} & P & {\F(P)} \\
			{T_P} & P & {F_P}
			\arrow["k", from=1-1, to=1-2]
			\arrow[equal, from=1-1, to=2-1]
			\arrow["{\eta_P}", from=1-2, to=1-3]
			\arrow[equal, from=1-2, to=2-2]
			\arrow["\theta", from=1-3, to=2-3]
			\arrow["k"', from=2-1, to=2-2]
			\arrow["q"', from=2-2, to=2-3]
		\end{tikzcd}\]
		we deduce that $\theta$ is a monomorphism. Therefore, $\theta$ is an isomorphism. To conclude, observe that $\pi_A \theta \eta_P = \pi_A q = \F(g') \eta_P$. Since $\eta_P$ is an epimorphism, it follows that $\pi_A \theta = \F(g')$ and, similarly, $\pi_C \theta = \F(p')$. Therefore $\F$ is protoadditive.
	\end{proof}
	
	We observe that, if $\tort$ is hereditary with respect to the class of protosplit monomorphisms, it follows easily that $\tort_1$ is also hereditary with respect to protosplit monomorphisms. Applying the previous results, we get the following:
	
	\begin{prop}
		Suppose that $\torsione$ satisfies Conditions (M), (N) and (P). Then $(\tort_1, \torf_1)$ satisfies (P), too.
	\end{prop}
	
	Summing up, we showed that, given a torsion theory $\torsione$ in $\mathcal{C}$ satisfying Conditions (N) and (M), we are able to build a torsion theory $(\tort_1, \torf_1)$ in $\arrc$ which satisfies the same conditions. In this situation, if moreover $\torsione$ satisfies Condition (S) (resp. (P)), then $(\tort_1, \torf_1)$ satisfies Condition (S) (resp. (P)), too. Then, the process can be iterated, getting a sequence of torsion theories $(\tort_n, \torf_n)$ in $\mathrm{Arr}^n(\mathcal{C})$, all of which keep the same good properties of the original torsion theory $\torsione$.

	\section*{Acknowledgements}
	Both authors are members of the Gruppo Nazionale per le Strutture Algebriche, Geometriche e le loro Applicazioni (GNSAGA) dell'Istituto Nazionale di Alta Matematica ``Francesco Severi''.
	
	This work was supported by Shota Rustaveli National Science Foundation of Georgia (SRNSFG), through grant FR-24-9660.

\end{document}